\documentclass[12pt,epsfig,amsfonts]{amsart}
\usepackage{amsmath,amsthm,amssymb,amscd,epsfig,color}
\usepackage{ulem}
\usepackage[all]{xy}
\usepackage{graphicx}
\usepackage{mathrsfs}
\usepackage{ulem}
\usepackage{pb-diagram} 

\newcommand{\ve}{\varepsilon}
\newcommand{\spe}{\mathrm{sp}}

\newtheorem{prop}{Proposition}[section]
\newtheorem{lemma}[prop]{Lemma}

\newtheorem{thm}[prop]{Theorem}
\newtheorem{cor}[prop]{Corollary}

\theoremstyle{definition}
\newtheorem{definition}[prop]{Definition}

\theoremstyle{remark}
\newtheorem{remark}[prop]{Remark}

\numberwithin{equation}{section}

\newcommand\bE{{\mathbb E}}

\newcommand\bI{{\mathbb I}}

\newcommand\bN{{\mathbb N}}

\newcommand\bP{{\mathbb P}}

\newcommand\bR{{\mathbb R}}

\newcommand\bZ{{\mathbb Z}}

\newcommand\cA{{\mathcal A}}
\newcommand\cB{{\mathcal B}}

\newcommand\cD{{\mathcal D}}

\newcommand\cF{{\mathcal F}}

\newcommand\cI{{\mathcal I}}

\newcommand\cL{{\mathcal L}}

\newcommand\cP{{\mathcal P}}

\newcommand\cS{{\mathcal S}}

\newcommand\cW{{\mathcal W}}

\newcommand\fD{{\mathfrak D}}

\newcommand{\hr}{{\hat{r}}}
\newcommand{\hh}{{\hat{h}}}

\newcommand{\Lip}{{\mathrm{Lip}}}
\newcommand{\LL}{{\mathrm{LL}}}
\DeclareMathOperator{\diam}{diam}
\usepackage{dirtytalk}

\begin{document}

\author{Juho Lepp\"anen and Hiroki Takahasi}

\address{Department of Mathematics, Tokai University, Kanagawa, 259-1292, JAPAN} 
\email{leppanen.juho.heikki.g@tokai.ac.jp}

\address{Keio Institute of Pure and Applied Sciences (KiPAS), Department of Mathematics,
Keio University, Yokohama,
223-8522, JAPAN} 
\email{hiroki@math.keio.ac.jp}

\subjclass[2020]{Primary 37A25, 37A40; Secondary 37A55}
\thanks{{\it Keywords}: piecewise affine map, nonstationary dynamical system, decay of correlations, multivariate normal approximation}


\title[Nonstationary heterochaos baker maps]
 {Decay of correlations and normal approximation for
nonstationary heterochaos baker maps}
 \maketitle

 \begin{abstract}
 We study statistical properties of nonstationary 
 compositions of a sequence of heterochaos baker maps. 
For maps whose central direction is mostly expanding, we establish three main results: exponential rate of memory loss
for a broad class of measures (Theorem~\ref{thm:main-ml}), a functional correlation bound with stretched exponential decay
(Theorem~\ref{thm:main-fcb}), and an error bound 
for the multivariate central limit theorem
(CLT) in the Wasserstein-1 distance (Theorem~\ref{thm:main-clt}). The proofs are based on the construction of suitable
Gibbs--Markov induced maps for the nonstationary compositions.
In a certain special case we obtain a concrete estimate on the
rate of convergence in the multivariate CLT (Theorem~\ref{thm:nearby}).
 \end{abstract}

\tableofcontents

\newpage
\section{Introduction}

Statistical and stochastic properties of nonstationary dynamical systems have attracted considerable attention in recent years. 
Unlike stationary dynamical systems generated by a single transformation, nonstationary  dynamical systems are described by sequential compositions
$
T_{1,n} :=
T_n\circ\cdots\circ T_1,
$
where the individual maps $T_k\colon X \to X$ governing the time evolution 
of states
are allowed to vary with time. Such systems provide natural models for physical processes evolving in temporally changing environments and describe dynamics out of equilibrium. 
If the initial state $x\in X$ is sampled randomly, then in general the time
series $(T_{1,n}(x))_{n\geq 1}$ of states is a nonstationary process, even
when all the maps $T_k$ preserve a common measure. 
In the special case where the dynamical rules evolve according to a stochastic process, which is commonly assumed to be generated by an ergodic measure-preserving transformation, the system is called a random dynamical system.

Topics of sustained interest in the study of statistical properties of nonstationary dynamical systems include:
\begin{itemize}
    \item rates of memory loss and decay of correlations \cite{GO13, MO14, SYZ13}, 
    \item central limit theorems (CLTs) and invariance principles \cite{CR07,HNTV17}, 
    \item concentration inequalities \cite{AimRou,CC25,KL21},
    \item extreme value laws \cite{FFV17}, and
    \item linear response \cite{DGS23,DGTS25,GL26}.
\end{itemize}
A variety of techniques from the stationary setting, including spectral methods,
projective-metric methods, and coupling methods, have been adapted to obtain
memory loss and correlation decay bounds in the nonstationary setting.
For uniformly expanding and hyperbolic systems, these methods have also led to
a well-developed theory of statistical limit laws, including CLTs with explicit
convergence rates \cite{DM22,DL25,DH25,LW24}.
By comparison, quantitative limit theorems for nonstationary systems beyond
the uniformly hyperbolic setting are less developed.
Results in this direction include limit theorems for sequential and random
intermittent maps and for random nonuniformly hyperbolic systems
admitting tower extensions \cite{DL26,FFV18,H22,NTV18,S22}.
In the partially hyperbolic setting, quenched exponential decay of correlations \cite{NW15}, central and local limit theorems \cite{CNW23}, and linear and higher-order response \cite{CN24, DH26} have been established for certain random skew products.

The purpose of this paper is to investigate statistical properties of
nonstationary dynamical systems generated by heterochaos baker maps.
Introduced in 
\cite{STY21}, and later in \cite{TY23} in a more general form, heterochaos baker
maps form a simple model of partially hyperbolic systems. They are a
parametrized family of piecewise affine skew product maps on the unit
square with a uniformly expanding direction and a central direction in
which contracting and expanding behavior coexist. This structure gives
rise to several complex phenomena of partially hyperbolic dynamics,
including the coexistence of invariant sets with different unstable
dimensions, while remaining sufficiently explicit to permit a detailed
mathematical analysis. Heterochaos baker maps also provide natural
geometric models of the Dyck system \cite[Section~4]{Kri74}, thereby establishing
a remarkable connection between partially hyperbolic dynamics and
symbolic dynamics \cite{TY23}.

We focus on deterministic compositions of heterochaos baker maps whose parameters vary with time within a regime where the central direction is mostly expanding. We make no assumptions on the statistical properties of the parameter sequence. We establish quantitative statistical properties that hold uniformly over all admissible sequences of parameters, thereby providing a new class of nonstationary dynamical systems with such properties.

Our first main result establishes exponential memory loss for sufficiently regular initial distributions. More precisely, for any two Borel probability measures $\mu$ and $\nu$ that are absolutely continuous with respect to the Lebesgue measure and have H\"older continuous densities, we establish
\begin{align}\label{eq:ml_intro}
| (T_{1,n})_*\mu -  (T_{1,n})_*\nu | = O(s^n),
\end{align}
where $s \in (0,1)$ and $|\cdot|$ denotes the total variation of signed measures. This indicates that the initial state of the system is forgotten exponentially fast in a statistical sense. Such memory loss is a fundamental statistical property from which many other statistical limit theorems may be derived. Note that \eqref{eq:ml_intro} implies the exponential decay of correlations for H\"older continuous observables.

The proof of \eqref{eq:ml_intro} is based on a coupling argument as in
\cite{KKM19}, which in turn relies on the existence of suitable
Gibbs--Markov structures. A principal technical contribution of this paper is the construction of 
the associated
induced maps in the nonstationary setting together with uniform estimates on their inducing times. Our inducing scheme extends the Gibbs--Markov structure introduced in \cite{T25} in the stationary setting to arbitrary time-dependent compositions of heterochaos baker maps. In particular, \eqref{eq:ml_intro} extends 
the exponential decay of correlations in
\cite{T25} for heterochaos baker maps with mostly expanding center to 
nonstationary compositions. 
In the stationary setting, the decay of correlations for heterochaos baker maps has been analyzed in a broader parameter range \cite{AB26, T25,TT25}, including parameters where the central direction is mostly neutral or mostly contracting.

Leveraging the Gibbs--Markov structure, we establish a functional correlation
bound with stretched exponential decay for the nonstationary system.
Functional correlation bounds generalize classical pair and multiple
correlation bounds by allowing observables to depend on finite fragments
of the system's trajectory. Such bounds have previously been established for 
various uniformly and nonuniformly expanding/hyperbolic systems \cite{FV22, LNN25, LNN26}.
We use the functional correlation bound, together with Stein's method
as in \cite{LS20}, to
derive error bounds in the Wasserstein-1 distance for the multivariate CLT.
Under the assumption that the smallest eigenvalue of the covariance matrix grows linearly with time, we obtain the convergence rate
$
O(N^{-1/2}\log^5 N)
$ in the multivariate CLT.
For sequences of maps belonging to a sufficiently small neighborhood of a fixed heterochaos baker map, we show that this covariance growth condition follows from a nondegeneracy assumption on the limiting covariance for the fixed map.

\subsection*{Organization of the paper}
The remainder of the paper is organized as follows. In Sections~\ref{HC-def-sec} and~\ref{sec:results} we introduce the heterochaos baker maps and state our main results concerning the decay of correlations and 
normal approximation. Section~\ref{sec:dynamics} is devoted to the analysis of the dynamics of nonstationary heterochaos baker maps. It contains one of the principal technical results of the paper (Proposition~\ref{induce-prop}), which establishes the existence of Gibbs--Markov induced maps together with estimates on the tails of their
inducing times. In Section~\ref{sec:memory_loss} we prove exponential memory loss for nonstationary heterochaos baker maps. The result is deduced from an abstract coupling theorem for nonstationary systems admitting a suitable Gibbs--Markov structure. Sections~\ref{sec:fcb_proof} and~\ref{sec:clt} are devoted to the proofs of the functional correlation bound and the error bound in the multivariate CLT, respectively. In Section~\ref{sec:nearby} we establish the linear growth of the smallest eigenvalue of the covariance matrix of Birkhoff sums for compositions of nearby maps. 
Appendix~\ref{sec:app_a} contains the proof of the abstract coupling theorem in Section~3.

\subsection{The heterochaos baker map}\label{HC-def-sec}
Let $M\geq2$ be an integer
and let $X=[0,1]^2$. 
Write $(x_u,x_c)$ 
for the coordinates on $X$. 
We introduce a heterochaos baker map $f_a\colon X\to X$, where $a\in(0,\frac{1}{M})$ is a parameter.
Define $\tau_a\colon[0,1]\to[0,1]$ by 
\[\tau_a(x_u)=\begin{cases}\vspace{1mm}
\displaystyle{\frac{x_u-(i-1)a}{a}}&\text{ on }[(i-1)a,ia),\ i=1,\ldots,M,\\ \displaystyle{\frac{x_u-Ma}{1-Ma}}&\text{ on }
[Ma,1].
\end{cases}\]
We introduce two alphabets consisting of $M$ symbols
\[D(\alpha)=\{\alpha_1,\ldots,\alpha_M\}\ \text{ and }\ D(\beta)=\{\beta_1,\ldots,\beta_M\},\] and set
  \[D=D(\alpha)\cup D(\beta).\]
 For each $\gamma\in D$ define a region $\Omega_a(\gamma)\subset X$ by
\[\Omega_a(\alpha_i)=\left[(i-1)a,ia\right)\times
\left[0,1\right]\ \text{ for }i=1,\ldots,M,\]
and
\[\Omega_a(\beta_i)=\begin{cases}
\vspace{1mm}\displaystyle{\left[Ma,1\right]\times
\left[\frac{i-1}{M },\frac{i }{M }\right)}&
\text{ for }i=1,\ldots,M-1,\\
\displaystyle{\left[Ma,1\right]\times
\left[\frac{i-1 }{M},1\right]}&\text{ for }i=M,
\end{cases}\]
and put
\[\Omega_a(\alpha)=\bigcup_{i=1}^M \Omega_{a}(\alpha_i)\ \text{ and }\ \Omega_a(\beta)=\bigcup_{i=1}^M \Omega_{a}(\beta_i).\]
Notice that these regions are pairwise disjoint and altogether cover $X$. 
Define $f_{a}\colon X\to X$ by
\[\begin{split}
  f_a(x_u,x_c)=
  \begin{cases}
\displaystyle{\left(\tau_a(x_u),\frac{x_c}{M}+\frac{i-1}{M}\right)}&\text{ on }\Omega_a(\alpha_i),\ i=1,\ldots,M,\\
   \displaystyle{\left (\tau_a(x_u),Mx_c-i+1\right)}&\text{ on }\Omega_a(\beta_i),\ 
   i=1,\ldots,M.
   \end{cases}
\end{split}\]

\begin{figure}
\begin{center}\includegraphics[height=4cm,width=11cm]
{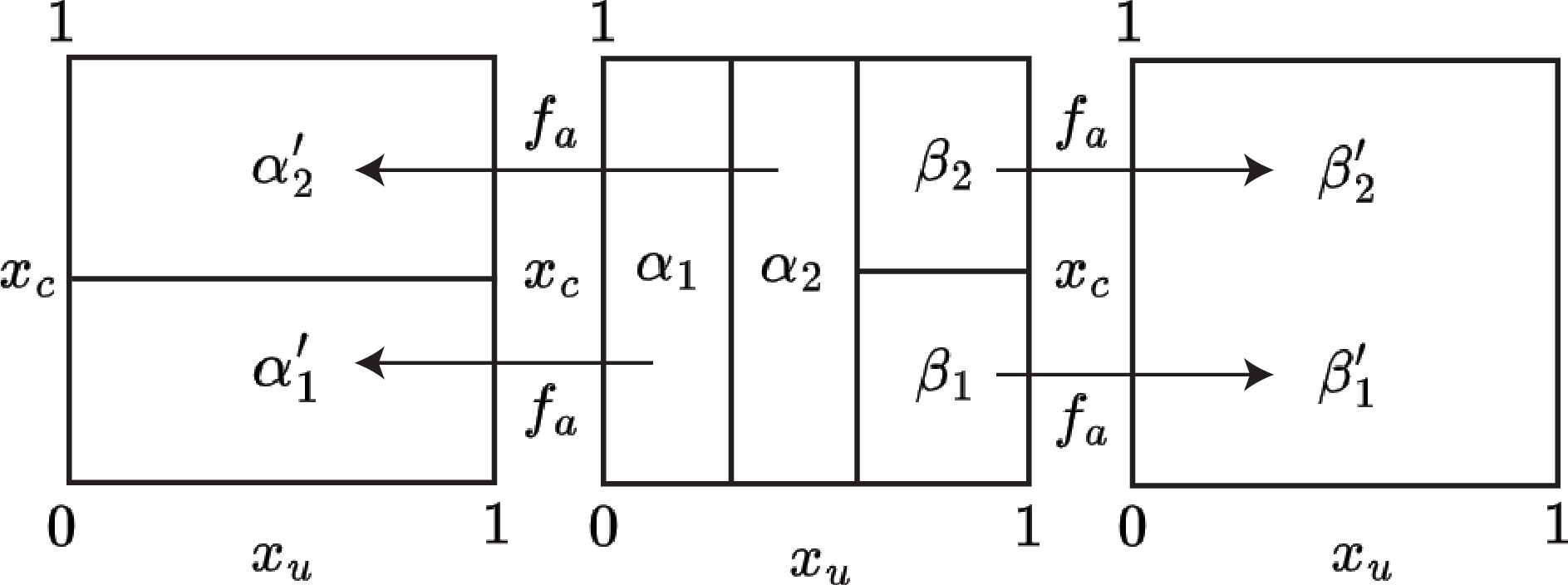}
\caption
{The heterochaos baker map with $M=2$. For each $\gamma\in D$,
the region $\Omega_a(\gamma)$  and its $f_a$-image are labeled with $\gamma$ and $\gamma'$ respectively: $f_a(\Omega_a(\beta_1) )=[0,1]\times[0,1)$ and $f_a(\Omega_a(\beta_2) )=X$.}\label{fig1}
\end{center}
\end{figure}
See \textsc{Figure}~\ref{fig1} for the case $M=2$. The dynamics of $f_{a}$ is partially hyperbolic:
the $x_u$-direction is expanding by factor $\frac{1}{a}$ or $\frac{1}{1-Ma}$ 
and the $x_c$-direction is a center, contracting by factor $\frac{1}{M}$ on $\Omega_{a}(\alpha)$ and expanding by factor $M$ on $\Omega_{a}(\beta)$.
It is easy to see that $f_a$ leaves invariant the Lebesgue measure on $X$, which we denote by $m$.
 Given a sequence $(a_k)_{k\geq1}$ of parameters in $(0,\frac{1}{M})$, we will consider the associated nonstationary system $(f_{a_k})_{k\geq1}$.

Define $\phi_a\colon X\to\mathbb R$ by
  \begin{equation}\label{geometric-c}\phi_a(x)=\begin{cases}-1&\text{ for }  x\in\Omega_a(\alpha),\\
 1&\text{ for }  x\in\Omega_{a}(\beta).\end{cases}\end{equation} 
Since $f_a$ is a skew product over $\tau_a$,
the ergodicity of $\tau_a$ with respect to the Lebesgue measure on $[0,1]$ implies that 
for $m$-almost every $x\in X$,
  \[\lim_{n\to\infty}
  \frac{1}{n}\sum_{k=0}^{n-1}\phi_a(f_a^k(x))=1-2Ma.\]
  Hence,
  the local stability in the $x_c$-direction along the orbit $(f_a^n(x))_{n\geq0}$
is determined by the sign of this limit value.
  We classify $f_a$ into three types:
\begin{itemize}
\item $a\in(0,\frac{1}{2M})$ (mostly expanding center);

 \item $a\in(\frac{1}{2M},\frac{1}{M})$
 (mostly contracting center);
 \item $a=\frac{1}{2M}$ (mostly neutral center).
 \end{itemize}
The parameter $a=\frac{1}{2M}$ is a bifurcation point at which the ``central Lyapunov exponent'' crosses zero.
For an individual $f_a$, 
 the decay of correlations for H\"older continuous observables with respect to $m$ was investigated in \cite{T25,TT25}.
 For $a\neq\frac{1}{2M}$, an exponential decay was shown 
in \cite{T25} by constructing a tower with exponential tail and applying the abstract result of Young \cite{You98}. 
For $a=\frac{1}{2M}$, the decay rate was shown to be polynomial \cite{TT25} with the optimal rate $N^{-3/2}$.

 \subsection{Statements of results}\label{sec:results}
 In our results stated below,
 we consider nonstationary heterochaos baker maps that consist only of maps with mostly expanding center.
Let $(a_k)_{k\geq1}$ be a sequence of parameters in $(0,\frac{1}{2M})$. For $k,\ell\in\mathbb N$ with $k\leq\ell $ we write \begin{equation}\label{map-notation} T_k=f_{a_k},\ 
T_{k,\ell}=T_{\ell}\circ\cdots\circ T_k.\end{equation}
Notice that $T_{k,k}=T_k$. Put
\[
\underline a=\inf_{k\geq1}a_k\ \text{ and }\ 
\overline a=\sup_{k\geq1}a_k.
\]
We say $(a_k)_{k\geq1}$ is {\it admissible} if 
\begin{equation}\label{eq:param_range}\max\left\{0,\frac{4M\overline a-1}{4M^2\overline a}\right\}<\underline a\leq\overline a<\frac{1}{2M}.\end{equation}

 Under condition \eqref{eq:param_range}, note that $\underline{a}\to\frac{1}{2M}$ as $\overline{a}\to\frac{1}{2M}$. In other words, the admissible parameter range shrinks to a point as coming closer to the bifurcation point $\frac{1}{2M}$. The first inequality in \eqref{eq:param_range} does not imply any restriction on $\underline{a}$ if  $\overline{a}\le\frac{1}{4M}$.
 It is easy to see that any constant sequence in $(0,\frac{1}{2M})$ is admissible: if $\underline a=\overline a<\frac{1}{2M}$ then 
 $\frac{4M\overline a-1}{4M^2\overline a}<\overline a$, and so the first inequality in \eqref{eq:param_range} holds.
In our results below we assume  $(a_k)_{k\geq1}$ is admissible.
 
All probability measures on $X$ we consider are Borel and assumed to be absolutely continuous with respect to $m$.  
Let $d\in\mathbb N$. 
For $x,x'\in\mathbb R^d$ let $|x-x'|$ denote the Euclidean distance between them.
For $g \colon X \to \bR^d$ and $\theta \in (0,1]$, we define
\[\|g\|_\infty=\sup_{x\in X}|g(x)|\ \text{ and }\ 
| g |_{\theta} = \sup_{\substack{x,x'\in X\\ x \neq x'}}
\frac{|g(x)-g(x')|}{|x-x'|^\theta},
\]
and set
\[\| g \|_{\theta} = \| g \|_\infty + | g |_{\theta}.\]
We also write
$| g |_{\Lip} = | g |_{1}$ and 
$\| g \|_{\Lip} = \| g \|_{1}$.

\begin{thm}[Exponential memory loss]\label{thm:main-ml}
Let $\mu_1$, $\mu_2$ be probability measures on $X$
satisfying $| d\mu_i/dm|_\theta \le L$ for $i=1,2$ for some 
$\theta \in (0,1]$ and $L \ge 0$.
Then, for any $n \ge 1$,
\[
| ( T_{1,n} )_* \mu_1 - ( T_{1,n} )_* \mu_2 |
\le C_L e^{-C' n}.\]
Here 
$C' > 0$ depends only on $\theta$ and $M, \overline{a}, \underline{a}$,
and $C_L$ depends in addition on $L$.
\end{thm}

For the remainder of this section, we fix a probability measure $\mu$ on $X$
 whose density $\rho$ is Lipschitz continuous with
\[\underline{\rho} = \inf_{x \in X} \rho(x) > 0.\]
For a function $F\colon X^n \to \bR$ and $\theta \in (0,1]$, where $n \ge 1$, we define
\[
|F|_{\theta}
= \max_{1 \le i \le n}
\sup_{x \in X^n}
\sup_{\substack{t,t'\in X\\ t \neq t'  }}
\frac{|F(x(t/i)) - F(x(t'/i))|}{|t-t'|^\theta}.
\]
Here $x(t/i) \in X^n$ denotes the vector obtained from
$x \in X^n$ by replacing its $i$th component with $t \in X$.
We say $F$ is {\it separately $\theta$-H\"older continuous} if
$|F|_{\theta} < \infty$, and in this case we define
\[
\|F\|_{\theta} = \|F\|_\infty + |F|_{\theta}.
\]

\begin{thm}[Functional correlation bound]\label{thm:main-fcb} 
There exist constants $C,C' > 0$ depending only on $\theta, \underline{\rho}, | \rho |_{ \Lip }$ 
and  $M, \overline{a}, \underline{a}$,
such that for any integer $n \ge 2$, 
any
separately $\theta$-H\"older continuous function $F \colon X^n \to \bR$, any integer $1\leq l<n$ and any sequence 
$0 \le i_1 < \cdots < i_n$ of integers, we have
\[
\begin{split}
&\biggl| \int_X H(x,x) \,  d \mu(x) - \iint_{X^2}  H(x,y) \, d \mu(x) \, d \mu(y)  \biggr| 
   \le  C \Vert F \Vert_{\theta} ( i_l + 1)  e^{ - C' \sqrt{i_{l + 1} - i_l}  },
\end{split}
\]
where $H\colon X^2\to\mathbb R$ is given by
\[
H(x,y) = F(  T_{1, i_1}(x), \ldots, T_{1, i_l}(x),  T_{1, i_{l + 1 }}(y), \ldots, 
 T_{1, i_{ n }}(y)  ).\]
\end{thm}
We emphasize that the constants $C, C'$ in 
Theorem~\ref{thm:main-fcb}  are independent of all the indices $n,l,i_1,\ldots,i_n$. This uniformity is important for the application of Theorem~\ref{thm:main-fcb} to the multivariate CLT rates below.
The proof of Theorem~\ref{thm:main-fcb} is roughly based on the approach in 
\cite[Section~6]{LNN25}, where a bound similar in spirit 
was established for a class of nonuniformly expanding interval maps. 
In contrast to the setting of \cite{LNN25}, heterochaos baker maps do not expand the Euclidean distance.
To overcome this difficulty, we use the Gibbs--Markov structure established in Proposition~\ref{induce-prop}.
This leads to the stretched exponential rate 
in Theorem~\ref{thm:main-fcb}, 
which is likely suboptimal.
 In contrast to \cite{LNN25}, the Gibbs-Markov structure established in Proposition~\ref{induce-prop} is not associated with any first return map.

Let 
$(g_k)_{k \ge 0}$ be a sequence of functions
$g_k \colon X \to \bR^d$ such that
\[
\sup_{k \ge 0} \| g_k \|_{\theta} \le L,
\]
for some $\theta\in(0,1]$ and $L \ge 1$. 
For each $k \in \bZ_+ = \{0,1,\ldots\}$ we set 
\[ \bar g_k = g_k - \int_X g_k \circ T_{1,k} \, d \mu, \quad
\xi_k = \bar g_k \circ T_{1,k},
\]
and consider $(\xi_k)_{k\geq0}$ as a stochastic process on $(X,\mu)$. 
We apply Theorem~\ref{thm:main-fcb} to establish
an error bound for the multivariate CLT in the Wasserstein-$1$ distance.
Before stating the result, we introduce
some further notation. For $N \in \mathbb N$, 
let
\begin{align*}
	S_N &=  \sum_{k=0}^{N-1} \xi_k, \quad 
	\Sigma_N = \int_X S_N S_N^{ \mathrm{T} } \, d \mu, \quad W_N = \Sigma_N^{-1/2} S_N,
\end{align*}
where we assume that the covariance matrix
$\Sigma_N \in \bR^{d \times d}$ is positive definite. We write $\lambda_{\min}(N)$ for the least eigenvalue of 
$\Sigma_N$. 

Denote by $N_d(0, \mathbb{I}_{d})$ the $d$-dimensional normal distribution with mean zero and 
covariance matrix $\mathbb{I}_{d}$, where $\mathbb{I}_{d}$ denotes the $d \times d$ 
identity matrix, 
and write $\cL(\xi)$ for the law (distribution) of a random variable $\xi$. 
For $Z \sim N_d(0, \mathbb{I}_{d})$, we define the Wasserstein-$1$ distance between the laws of 
$W_N$ and $Z$ by
\[
d_{ \mathcal{W} }(  \cL( W_N ) ,  \cL( Z )  ) = \sup_{ h \in \mathcal{W} } \biggl| \int h(W_N) \, d \mu  - E( h (  Z )  )  \biggr|,\]
 where $\mathcal W$ denotes the class of all $1$-Lipschitz functions, namely
\[
\mathcal{W} = \{  h \colon\bR^d \to \bR \colon  | h(x) - h(y) | \le |x-y|   \text{ for all }x,y\in \mathbb R^d\}.\]

\begin{thm}[Error bound in the multivariate CLT]\label{thm:main-clt} There exists a constant $C > 0$ depending only on 
$\theta, \underline{\rho}, | \rho |_{ \Lip }$ and  $M, \overline{a}, \underline{a}$ such that 
\begin{align}\label{eq:w1_bound}
\begin{split}
d_{ \mathcal{W} }(  \cL( W_N ),  \cL(  Z )   ) 
\le& C L^3  \sqrt{d} \max\{1 , \log \lambda_{ \min } (N) \} \lambda_{ \min }^{-3/2} (N)  N (  1 + \log^4N ) \\
&+ C \sqrt{d} L \lambda_{ \min }^{-1/2}(N).
\end{split}
\end{align}
In particular, $W_N \stackrel{\cD}{\to} Z$ as $N \to \infty$, provided that \[\lambda_{\min}^{-1}(N) = o(N^{-2/3} \log^{-10/3} N  ),\]
where $\stackrel{\cD}{\to}$ denotes the convergence in distribution.
\end{thm}

Let $a \in (0, \frac{1}{2M})$.
We say a function $\varphi \colon X \to \bR^d$  is {\it not a coboundary for $f_a$ in any direction} if for any unit vector $v \in \bR^d$ it is not possible to write 
\[
 \langle v,\varphi\rangle = \psi_v - \psi_v \circ f_a + c_v,
\]
where $\psi_v \colon X \to \bR$ is in $L^2(m)$, 
 $c_v \in \bR$ and $\langle\cdot ,\cdot\rangle$ denotes the inner product on $\mathbb R^d$.
In a special case of maps nearby a fixed heterochaos baker map
$f_a$ and Lipschitz observables that are not a coboundary for $f_a$ 
in any direction, we establish the linear growth of $\lambda_{\min}(N)$
in Theorem~\ref{thm:nearby} below. This leads to a concrete estimate on the
rate of convergence in the multivariate CLT.

\begin{thm}[Rate in the multivariate CLT for nearby maps]\label{thm:nearby}
Let 
$a\in(0,\frac{1}{2M})$.
Let $\mu = m$, and let 
$g_k = g$ for all $k \ge 0$, where $g \colon X \to \bR^d$ is a Lipschitz continuous function 
that is not a coboundary for $f_a$ in any direction. 
There exists 
$\varepsilon\in(0,\min\{a,\frac{1}{2M}-a\})$
such that for any sequence $(a_k)_{k\geq1}$ of parameters 
in $(0,\frac{1}{2M})$ satisfying 
$|a_k -a|< \varepsilon$ for all $k\geq1$, 
\begin{align}\label{eq:lambda_growth}
\lambda_{\min}^{-1}(N) = O(N^{-1}),
\end{align}
as $N \to \infty$. In particular, 
\begin{align}\label{eq:rate_nearby}
&d_{ \mathcal{W} }(  \cL( W_N ),  \cL(  Z )   ) = O( N^{ -1/2 } \log^5 N  ).
\end{align}
\end{thm}

For simplicity,
 in all of these main results we have only treated sequences of parameters in the mostly expanding center regime $(0,\frac{1}{2M})$. Actually, 
 all the main results 
can be slightly generalized to include sequences that contain 
parameters not contained in $(0,\frac{1}{2M})$. For more details, see Remark~\ref{gen-rem}.

\section{Dynamics of nonstationary heterochaos baker maps}\label{sec:dynamics}
In this section we analyze the dynamics of nonstationary heterochaos baker maps.  
In Section~\ref{GM-sec} we state Proposition~\ref{induce-prop}, the principal result  on the existence of a Gibbs--Markov structure with exponential tail.
The rest of this section is entirely dedicated to a proof of Proposition~\ref{induce-prop}.

\subsection{Gibbs--Markov structure}\label{GM-sec}

We treat two-sided sequences of heterochaos baker maps.
As before, 
given a two-sided sequence $(a_k)_{k\in\mathbb Z}$ of parameters in the mostly expanding center regime $(0,\frac{1}{2M})$,
 for $k,\ell\in\mathbb Z$ with $k \le \ell$ write
\[T_k=f_{a_k},\ 
	T_{ k, \ell } = 
		T_\ell \circ \cdots \circ T_k.\]
Notice that $T_{k,k}=T_k$. For convenience we set $T_{k,k-1}={\rm id}_X$, the identity map on $X$.
 Let
${\rm int}(\cdot)$ denote the interior operation in $\mathbb R^2$.

\begin{prop}[Gibbs--Markov structure]
\label{induce-prop}
Let
 $(a_k)_{k\in\mathbb Z}$ be a sequence of parameters in $(0,\frac{1}{2M})$ satisfying 
 \begin{equation}\label{eq:param_range-z}\max\left\{0,\frac{4M\overline a-1}{4M^2\overline a }\right\}
 <
 \underline a\leq\overline a<\frac{1}{2M},\end{equation}
 where
 \[
\underline a=\inf_{k\in\mathbb Z}a_k\ \text{ and }\ 
\overline a=\sup_{k\in\mathbb Z}a_k.
\]
 There exist constants $\lambda>1$, $\delta_{\#}>0$, $A>0$, $A'>0$ depending only on $M, \overline{a}, \underline{a}$,  and 
a sequence $(\mathcal P_k,R_k)_{k\in\mathbb Z}$
such that the following hold:

\begin{itemize}
\item[(a)] $\cP_k$ is a countable Borel partition of 
	a subset $X_k = \bigcup_{W \in \cP_k} W$ of ${\rm int}(X)$ such that
	$m(W) > 0$ for all $W \in \cP_k$ and $m({\rm int}(X)\setminus X_k)=0$;  
	\item[(b)] $R_k \colon X \to [2, \infty]$ is a function that is constant on
	each $W \in \cP_k$ with value $R_k(W) \in \{  2, 3,\ldots \}$;
\item[(c)] 
for all $W\in\mathcal P_k$, the map $F_W\colon x\in W\mapsto T_{k,k+R_k(W)-1}(x)\in X$
is affine, 
$F_W(W)={\rm int}(X)$
and for all $x,x'\in W$, \[|F_W(x)-F_W(x')|\geq\lambda |x-x'|;\]
\item[(d)] 
for all $W\in\mathcal P_{k}$ and $x,x'\in W$, 
\[
	\max_{0 \leq j \leq R_k(W)} |T_{k,k+j-1}(x)- T_{k,k+j-1}(x')|
	\leq  |F_W(x)- F_W(x')|;\]

\item[(e)] 
$\inf_{k\in\mathbb Z}m(R_{k}=2)\geq\delta_{\#}$ and $\inf_{k\in\mathbb Z}m(R_{k}=3)\geq\delta_{\#}$;
\item[(f)] 
 for all $n\geq 1$,
 \[m(R_k\geq n)\leq A \exp( - A'  n ).\]
\end{itemize}
\end{prop}

In order to prove Proposition~\ref{induce-prop}, we start with an arbitrary one-sided sequence of heterochaos baker maps, and construct a partition $\mathcal P$ and an inducing time $R$ that depend on this sequence. We then apply
this construction to each truncated one-sided sequence $(T_{n})_{n\geq k}$, $k\in\mathbb Z$
to obtain $(\mathcal P_k,R_k)_{k\in\mathbb Z}$ with the required properties. 

The construction of $(\mathcal P,R)$ was successfully undertaken in \cite{T25} in the stationary case. 
We slightly modify the construction in \cite{T25} to cover the nonstationary case.
Throughout Sections~\ref{graph-sec-s} to \ref{tail-sec}, we assume
 $(b_k)_{k\geq1}$ is a sequence of parameters in $(0,\frac{1}{2M})$, and set
 \[
\underline b=\inf_{k\geq1}b_k, \  
\overline b=\sup_{k\geq1}b_k.
\]
 With a slight abuse of notation we write  \[T_k=f_{b_k},\ \phi_{k}=\phi_{b_k},\ \Omega_k(\gamma)=\Omega_{b_k}(\gamma)\text{  for }\gamma\in D\text{ and }\gamma\in\{\alpha,\beta\}.\] 
 For compositions of the above heterochaos baker maps in $(T_k)_{k\geq1}$, we keep the notation in \eqref{map-notation}, and set $T_{k,k-1}={\rm id}_{X}$.

 In Section~\ref{graph-sec-s}
 we introduce a graph representation of the dynamics of nonstationary heterochaos baker maps in the spirit of Markov diagram \cite{Hof79}, and provide an upper bound on the number of paths in this diagram with prescribed properties. In Section~\ref{the-stop} we introduce an inducing time $R$ and analyze its properties. In Section~\ref{ind-exp-sec} we construct a partition $\mathcal P$, and in Section~\ref{tail-sec} show that 
 the tail $m(R\geq n)$ decays exponentially in $n$. In Section~\ref{pf-sec} we complete the proof of Proposition~\ref{induce-prop}.

 \subsection{Graph representation of  dynamics}\label{graph-sec-s}


Recall that $D=D(\alpha)\cup D(\beta)$. 
Let $A$ be a non-empty subset of $X$ with $A\subset\Omega_{n}(\gamma)$ for some $n\geq1$ and $\gamma\in D$.
A non-empty subset $B$ of $X$ is called {\it an $n$-successor} of $A$ if there exists $\gamma'\in D$ such that $B=T_{n}(A)\cap{\rm int}(\Omega_{n+1}(\gamma'))$.
We set \[\mathcal V(\alpha)=\{{\rm int}(\Omega_{1}(\alpha_i))\colon 1\leq i\leq M \},\ \mathcal V(\beta)=\{{\rm int}(\Omega_{1}(\beta_i))\colon 1\leq i\leq M \},\]
 and set
 \[\mathcal V_{1}=\mathcal V(\alpha)\cup\mathcal V(\beta).\] 
 Define $\mathcal V_{n}$, $n=2,3,\ldots$ by the recursion formula \[\mathcal V_{n}=  \mathcal V_{n-1}\cup\{B\colon \text{$B$ is an $(n-1)$-successor of an element of $\mathcal V_{n-1}$}\}.\] 
Put $\mathcal V_{0}=\emptyset$ for convenience.
Set \[\mathcal V=\bigcup_{n\geq1}\mathcal V_n.\] 
Let $u,v\in\mathcal V$ satisfy 
$u\in\mathcal V_{k }\setminus \mathcal V_{k-1}$ and 
$v\in\mathcal V_{n}\setminus \mathcal V_{n-1}$ for some $k,n\in\mathbb N$
with $|k-n|\leq 1$.
Let $\gamma\in D$ satisfy $v\subset{\rm int}(\Omega_{n}(\gamma))$.
If 
\[v=T_{k}(u)\cap{\rm int}(\Omega_{n}(\gamma)),\]
then we write $u\to v$. 
For example, if $n\geq2$, $u\in\mathcal V_{n-1}\setminus\mathcal V_{n-2}$ and $v\in\mathcal V_n$ is an $(n-1)$-successor of $u$ then $u\to v$.
We consider the directed graph $(\mathcal V,\to)$.

To each $v\in\mathcal V_{n}\setminus \mathcal V_{n-1}$, $n\in\mathbb N$ uniquely corresponds a string 
$v_1\cdots v_n$ of elements of $\mathcal V$ such that $v_j\in\mathcal V_{j}$, $v_j\to v_{j+1}$ for $1\leq j\leq n-1$ and $v_n=v$. Let $\Pi(v)=\omega_1\cdots\omega_n\in D^{n}$ 
denote the word from $D$ satisfying $v_j\subset\Omega_{j}(\omega_j)$ for $1\leq j\leq n$. 
Let $D^*$ denote the set of finite words from $D$, namely \[D^*=\bigcup_{n\geq1} D^{n}.\]
The map $\Pi\colon \mathcal V\to D^*$
is called an {\it address map}.
It is easy to see that
the address map is injective.

The set $\Pi(\mathcal V)$ does not depend on the parameter sequence $(b_k)_{k\geq1}$.
Indeed,
consider the monoid with $0$, with $2M$ generators in $D$ and the identity element
$1$ with relations 
\[\alpha_i\cdot\beta_j=\delta_{ij},\
0\cdot 0=0\text{ and }\] \[\gamma\cdot 1= 1\cdot\gamma=\gamma,\
 \gamma\cdot 0=0\cdot\gamma=0
\text{ for }\gamma\in \{ 1\}\cup D^*,\]
where $\delta_{ij}$ denotes Kronecker's delta.
Let ${\rm red}\colon D^*\to D^*\cup\{0,1\}$ denote the map
that assigns the reduced word in $D^*$, or $0$ or $1$:
for $n\in\mathbb N$ and $\omega_1\cdots\omega_n\in D^*$,
${\rm red}(\omega_1\cdots\omega_n)$ is the shortest word
in $D^*$, or $0$ or $1$ obtained by applying the above relations repeatedly to $\omega_1\cdots\omega_n$. 
We have
\begin{equation}\label{identity}\Pi(\mathcal V)=\{\omega\in D^*\colon{\rm red}(\omega)\neq0\}.\end{equation}
The right set in \eqref{identity} is the set of finite words that appear in some elements of the Dyck shift \cite[Section~4]{Kri74}.
In the case $\underline{b}=\overline{b}$,  \eqref{identity} follows from the result in \cite{TY23}, and 
the directed graph $(\mathcal V,\to)$ is called the {\it Markov diagram} for  $f_{\underline{b}}$ (see \cite[Section~2.1]{T25}). Part of its structure is depicted in \textsc{Figure}~\ref{diagram}.

\begin{figure} \begin{center}
\includegraphics[height=8cm,width=14.5cm]
{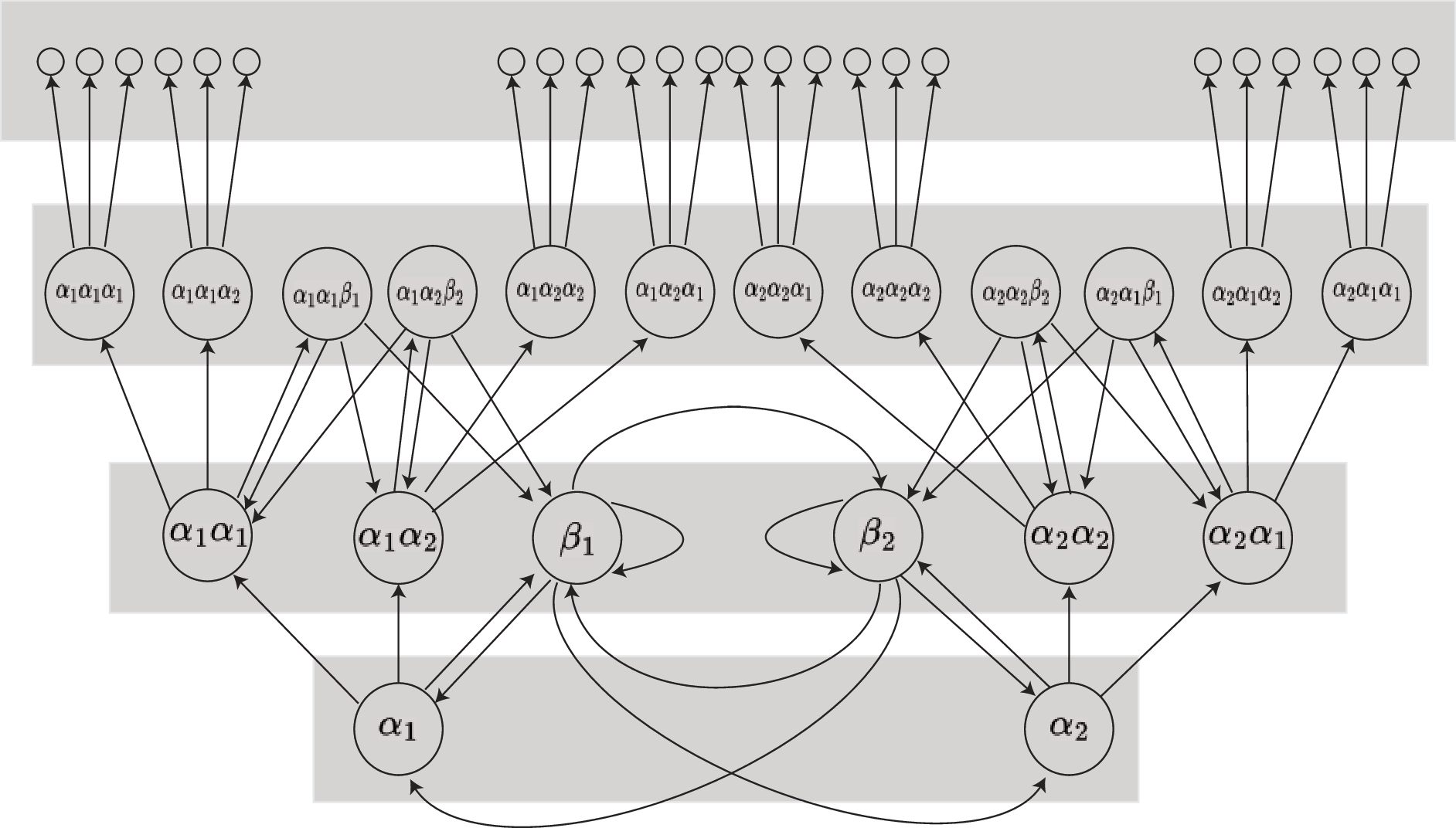}
\caption{Part of the Markov diagram $(\mathcal V,\to)$ for $f_a$, $a\in(0,\frac{1}{M})$ with $M=2$.
The vertices ${\rm int}(\Omega_a(\omega_1))$,  $f_a({\rm int}(\Omega_a(\omega_1 ))\cap{\rm int}(\Omega_a(\omega_2) )$,
$f_a(f_a({\rm int}(\Omega_a(\omega_1 ))\cap {\rm int}(\Omega_a(\omega_2)))\cap {\rm int}(\Omega_a(\omega_3))$ in $\mathcal V$ 
are labeled with $\omega_1$, $\omega_1\omega_2$, $\omega_1\omega_2\omega_3$ respectively.}
\label{diagram}
\end{center}
\end{figure}

Even in the case
$\underline{b}<\overline{b}$,
tracing the argument in \cite{TY23} one can verify \eqref{identity}.
However, then
the directed graph $(\mathcal V,\to)$ is too fine: for example, $\Omega_{1}(\alpha_1)$ and $\Omega_{2}(\alpha_1)$ become different elements of $\mathcal V$ when $b_1\neq b_2$.
To resolve this problem we introduce a quotient graph. 
For each $\omega\in \Pi(\mathcal V)$ there is a decomposition
\[{\rm red}(\omega)=\omega^+\omega^-,\]
where $D(\alpha)^0:=\{1\}$, $D(\beta)^0:=\{1\}$, $\omega^+\in \bigcup_{n\geq0}D(\beta)^n$
and $\omega^-\in \bigcup_{n\geq0}D(\alpha)^n$ (see \cite{HI05}).
Let $\omega^+_r\in\{1\}\cup D(\beta)$ denote the rightmost symbol of the word $\omega^+$.
For $u,v\in\mathcal V$
we set $u\sim v$ if 
$\omega=\Pi(u)$, $\eta=\Pi(v)$,
$\omega^+_r=\eta^+_r$ and
$\omega^-=\eta^-$.
It is easy to check that $\sim$ defines an equivalence relation on $\mathcal V$.
For each $v\in \mathcal V$, let $[v]$ denote the equivalence class of $v$.
For any subset $\mathcal V'$ of $\mathcal V$, set $[\mathcal V']=\{[v]\colon v\in\mathcal V'\}$. 
We write $[u]\to [v]$ if 
there exist $u'\in[u]$ and $v'\in[v]$ such that $u'\to v'$.
The directed graph $([\mathcal V],\to)$ is called a {\it nonstationary Markov diagram} for $(T_{1,n})_{n\geq1}$.

\begin{remark}
In the case $\underline b=\overline b$,
for $u,v\in\mathcal V$ we have $u\sim v$ if and only if $u=v$. In this sense,
the nonstationary Markov diagram is a generalization of the Markov diagram to the nonstationary setting.
\end{remark}

For each $n\in\mathbb N$,
let $P(n)$ denote the set of 
$n$-strings $[v_1]\cdots [v_n]$ of elements of $[\mathcal V]$ such that 
$[v_{j}]\to [v_{j+1}]$ holds for all $1\leq j\leq n-1$, and
there is no
 $j\in\{1,\ldots,n-1\}$ such that
  $[v_{j}]\in[\mathcal V(\beta)]$ and $[v_{j+1}]\in[\mathcal V(\beta)]$. 
 Define a subset $P^*(n)$ of $P(n)$ by
  \[P^*(n)=\begin{cases}\{\![v_1]\cdots [v_{n}]\in P(n)\colon [v_1]\in[\mathcal V(\alpha)]\text{ and }[v_{n}]\in  [\mathcal V_{1}]\setminus [\mathcal V(\alpha)]\}&\!\!\text{if $n$ is odd,}\\\{\![v_1]\cdots [v_{n}]\in P(n)\colon [v_1], [v_{n}]\in[\mathcal V_{1}]\setminus [\mathcal V(\alpha)] \}&\!\!\text{if $n$ is even.}\end{cases}\]
   \begin{lemma}\label{path-number} For all $n\geq3$ we have      
        \[\# P^*(n)\leq
       \frac{2M}{n}\begin{pmatrix}n+1\\\lfloor\frac{n+2}{2}\rfloor\end{pmatrix}M^{\lfloor\frac{n}{2}\rfloor}.\]
   \end{lemma}\begin{proof}
Fix a reference parameter $a\in(0,\frac{1}{M})$, and let $(\mathcal V_{a},\to)$ be the Markov diagram for $f_{a}$ with the address map $\Pi_a\colon \mathcal V_a\to D^*$.
Define a map $\Phi\colon[\mathcal V]\to\mathcal V_{a}$ by
\[\Pi_a(\Phi([v]))=\omega^+_r\omega^-\ \text{ for }
v\in\mathcal V,\ \omega=\Pi(v).\]
Then $\Phi$ is a well-defined injection.
Moreover, $[u],[v]\in[\mathcal V]$, $[u]\to [v]$ implies $\Phi([u])\to \Phi([v])$.
   Meanwhile, 
   the desired upper bound in Lemma~\ref{path-number} was established in \cite[Lemma~2.9]{T25} when
   $\underline{b}=\overline{b}$. By virtue of the above isomorphism $\Phi$, it remains valid when $\underline{b}<\overline{b}$. 
\end{proof}

\textcolor{red}{
}


\subsection{The inducing time}\label{the-stop}
By a {\it rectangle}
we mean a product of two 
non-degenerate intervals in $[0,1]$.
For a rectangle $B=B_u\times B_c\subset X$ we write \[
|B|_u=|B_u|,\  |B|_c=|B_c|.\]
Here, $|I|$ denotes the Euclidean length of an interval $I$.
For $x\in X$ and $n\in\mathbb N$, 
let $K_{n}(x)$ denote the maximal rectangle containing $x$ on which $T_{1,n}$ is affine.
For example, we have
 $|K_{1}(x)|_c=1$ if $x\in\Omega_{1}(\alpha)$ and
 $|K_{1}(x)|_c=\frac{1}{M}$ if $x\in\Omega_{1}(\beta)$.
We say an integer $n\geq2$ is a  {\it cutting time} of $x$
if 
\[\frac{|K_{n}(x)|_c}{|K_{n-1}(x)|_c}=\frac{1}{M}.\]
We define an {\it inducing time} $R\colon X\to[2,\infty]$ by
  \begin{equation}\label{afnew-eq} R(x)=\inf\{n\geq2\colon \text{$n$ is a cutting time of $x$}\}.\end{equation}

For $n\in\mathbb N$
define
\[\phi_{1,n}=\sum_{j=0}^{n-1}\phi_{j+1}\circ T_{1,j}.\] 
For $x\in X$ we have 
 \[\phi_{1,n}(x)=n-2\cdot\#\{0\leq j\leq n-1\colon T_{1,j}(x)\in\Omega_{j+1}(\alpha)\}.\]
 In other words, $(n-\phi_{1,n})/2$ counts the number of returns to the regions where the central direction is contracting. 
 Put $\phi_{1,0}\equiv 0$ for convenience.
  
For each $k\geq1$, put
   \[E_{k}=\Omega_{k}(\beta)\cap T_{k}^{-1}(\Omega_{k+1}(\beta)).\]
  The next proposition represents the inducing time in terms of return times to the sets $E_k$, $k\geq1$.
 \begin{prop}
\label{character}\
\begin{itemize}
\item[(a)] If $x\in\Omega_{1}(\alpha)$ then
\[R(x)=\inf\{n\geq 0
\colon \phi_{1,n}(x)=-1\text{ and }T_{1,n}(x)\in E_{n+1} \}+2.\]
\item[(b)] If $x\in\Omega_{1}(\beta)$ then
\[R(x)=\inf\{n\geq 0
\colon \phi_{1,n}(x)=0\text{ and }T_{1,n}(x)\in E_{n+1} \}+2.\]
\end{itemize}
\end{prop}
For a proof of this proposition we need the next lemma.
\begin{lemma}\label{chara-lem}
Let $x\in\Omega_{1}(\beta)$. An integer $n\geq2$ is a cutting time of $x$
if and only if 
$T_{1,n-2}(x)\in E_{n-1}$
and $\phi_{1,n-2}(x)\geq0$.
\end{lemma}
\begin{proof}
To show the `if' part, suppose
   $T_{1,n-2}(x)\in E_{n-1 }$
and $\phi_{1,n-2}(x)\geq0$.
There exists $i\in\{1,\ldots,M\}$ such that
$T_{1,n-2}(x)\in\Omega_{n-1}(\beta_i)$.
Let $B$ denote the connected component of $T^{-1}_{1,n-2}(\Omega_{n-1}(\beta_i))$ that contains $x$.
Then $B$ is a rectangle, $T_{1,n-2}|_B$ is affine, and  condition $\phi_{1,n-2}\geq0$ implies $|T_{1,n-2}(B)|_c=\frac{1}{M}$. Hence  $n$ is a cutting time of $x$.

To show the `only if' part, let $n\geq2$ be a cutting time of $x$. Then we have
      $T_{1,n-2}(x)\in E_{n-1}$. 
    Since $x\in\Omega_{1}(\beta)$, 
    it follows that $\phi_{1,n-2}(x)\geq0$.
\end{proof}
\if0\begin{lemma}\label{chara-lem}
Let $x\in\Omega^+_{k}(\beta)$ and let $n\geq2$ be an integer.
Then $n$ is a $k$-cutting time of $x$
if and only if 
$f_{k,k+n-3}(x)\in E_{k+n-2}$
and $\phi_{k,k+n-3}(x)\geq0$.
\end{lemma}
\begin{proof}
To show the `if' part, suppose
    $f_{k,k+n-3}(x)\in E_{k+n-2}$
and $\phi_{k,k+n-3}(x)\geq0$.
There exists $i\in\{1,\ldots,M\}$ such that
$f_{k,k+n-3}(x)\in\Omega_{k+n-2}^+(\beta_i)$.
Let $B$ denote the connected component of $f^{-1}_{k,k+n-3}(\Omega_{k+n-2}^+(\beta_i))$ that contains $x$.
Then $B$ is a rectangle, $f_{k,k+n-3}|_B$ is affine and $|f_{k,k+n-3}(B)|_c=\frac{1}{M}$. Hence  $n$ is a $k$-cutting time of $x$.
To show the `only if' part, let $n\geq2$ be a $k$-cutting time of $x$. The definition of $k$-cutting time implies
      $f_{k,k+n-3}(x)\in E_{k+n-2}$. 
    Since $x\in\Omega^+_{k}(\beta)$, applying
     by \eqref{interpret-eq2} to the path corresponding to ther nonstationary orbit $x,f_{k,k}(x),\ldots,f_{k,k+n-3}(x)$ we obtain $\phi_{k,k+n-3}(x)\geq0$.
\end{proof}\fi
\begin{proof}[Proof of Proposition~\ref{character}]
We treat three cases separately.
\smallskip

\noindent{\it Case~1: $x\in E_1\subset\Omega_1(\beta)$.}
Then $R(x)=2$ holds. Since $\phi_{1,0}\equiv0$ and 
$T_{1,0}(x)=x\in E_1$, we obtain the equality in (b) with $n=0$.
\smallskip

\noindent{\it Case~2: $x\in \Omega_{1}(\beta)\setminus E_{1}$.}
Then $R(x)\geq3$ holds.
We claim $\phi_{1,R(x)-2}(x)=0$, for otherwise 
the structure of the nonstationary Markov diagram described in Section~\ref{graph-sec-s} would imply
 $\phi_{1,R(x)-2}(x)>0$, and 
there would be $i\in\{2,\ldots,R(x)-1\}$ such that $T_{1,i-2}(x)\in 
E_{i-1}$ and $\phi_{1,i-2}(x)=0$. From Lemma~\ref{chara-lem}, $i$ would be a cutting time of $x$, a contradiction to the minimality in the definition of $R(x)$. This claim yields
\[R(x)\geq\inf\{n\geq 0
\colon \phi_{1,n}(x)=0\text{ and }T_{1,n}(x)\in E_{n+1} \}+2.\]
The reverse inequality follows from Lemma~\ref{chara-lem} and the minimality of 
$R$. 
We have established (b).\smallskip

\noindent{\it Case~3:
$x\in\Omega_1(\alpha)$.} 
Pick a parameter $b_0\in(0,\frac{1}{M})$
and set $T_0=f_{b_0}$, $\Omega_0(\beta)=\Omega_{b_0}(\beta)$.
There exists $y\in\Omega_{0}(\beta)\setminus 
(\Omega_{0}(\beta)\cap T_{0}^{-1}(\Omega_{1}(\beta)))$
such that $T_0(y)=x$.
The reasoning in Case~2 applied to $y$ yields
\[R(y)=\inf\{n\geq 0
\colon \phi_{1,n}(T_0(y))+\phi_{b_0}(y)=0\text{ and }T_{1,n}(T_0(y))\in E_{n+1} \}+2.\]
Hence we obtain
\[R(x)=R(y)-1=\inf\{n\geq 0
\colon \phi_{1,n}(x)=-1\text{ and }T_{1,n}(x)\in E_{n+1} \}+2,\]
as required.
\end{proof}

 \begin{lemma}\label{m-23}
We have \[m(R=2)=(1-Mb_1)(1-Mb_{2})\ \text{ and}\]
\[m(R=3)=Mb_1(1-Mb_2)(1-Mb_{3}).\]
\end{lemma}
\begin{proof}Since
  $\{R=2\}= E_1$
  and 
  $\{R=3\}=\Omega_{1}(\alpha)\cap T_1^{-1}(E_2)$, the desired equalities hold. \end{proof}
\begin{prop}\label{R-fin}
If $\overline b<\frac{1}{2M}$, then
$R(x)$ is finite
for $m$-almost every $x\in X$.
\end{prop}

\begin{proof}
 It is easy to see that
$\{\phi_{j+1}\circ T_{1,j}+2Mb_{j+1}\}_{j\geq 0 }$ is a sequence of independent random variables on the probability space
$(X,m)$ with expectation $1$ and uniformly bounded variances. By the law of large numbers, for $m$-almost every $x\in X$ we have
\[\lim_{n\to\infty}\frac{1}{n}\sum_{j=0}^{n-1}\left(\phi_{j+1}\circ T_{1,j}(x)+2Mb_{j+1}\right)=1,\]
and thus
\begin{equation}\label{1-NUE}\begin{split}\limsup_{n\to\infty}\frac{\phi_{1,n}(x)}{n}&=\limsup_{n\to\infty}\frac{1}{n}\sum_{j=0}^{n-1}\phi_{j+1}\circ T_{1,j}(x)\\
&\geq\left(1-2M\limsup_{n\to\infty}\frac{b_1+\cdots+b_{n}}{n}\right)>0.\end{split}\end{equation}
The last inequality follows from the assumption $\overline b<\frac{1}{2M}$. Meanwhile, Proposition~\ref{character} and 
the structure of the nonstationary Markov diagram described in Section~\ref{graph-sec-s} together imply that
$\limsup_{n\to\infty}\phi_{1,n}(x)/n\leq0$ holds
provided
$x\in X$ 
and $R(x)=\infty$.
 Hence the desired conclusion follows.
\end{proof}

To proceed, 
we need another representation of $R$.
Let $V$ be a non-empty subset of $X$ and
let $n\in\mathbb Z_+$.  A connected component of $T_{1,n}^{-1}(V)$ is called a {\it pullback} of $V$ by $T_{1,n}$. If $W$ is a pullback of $V$ by $T_{1,n}$ and $T_{1,n}|_W$ is affine, then $W$ is called {\it an affine pullback of $V$ by $T_{1,n}$}.
If $V$ is connected and $W$ is an affine pullback of $V$ by $T_{1,n}$, then $T_{1,n}(W)=V$.

The next proposition represents the inducing time in terms of pullbacks.

\begin{prop}
 \label{exist-cor}
 If $x\in{\rm int}(X)$,
$R(x)$ is finite and $T_{1,R(x)}(x)\in{\rm int}(X)$, then 
\[R(x)=\min\left\{\begin{split}&n\geq2\colon\text{there exists an affine pullback $B$ of ${\rm int}(X)$ by $T_{1,n}$}\\
&\quad\quad\quad\ \text{such that }
x\in B\text{ and }|B|_c=\frac{1}{M}|K_{1}(x)|_c\end{split}\right\}.\]
\end{prop}\begin{proof}
In the case
$R(x)=2$ we have $x\in E_{1}$, and the desired equality is obvious. Suppose $R(x)\geq3$. Then $x\notin E_{1}$ holds.
By Proposition~\ref{character}, there exists
 $i\in\{1,\ldots,M\}$ such that
$T_{1, R(x)-2}(x)\in\Omega_{R(x)-1}(\beta_i)\cap E_{R(x)-1 }$. 
    The pullback of the rectangle
$\Omega_{R(x)-1}(\beta_i)\cap E_{R(x)-1 }$
by $T_{1,R(x)-2}$
that contains $x$, denoted by $B'$, is an affine pullback and satisfies
 \[|B'|_c=|K_{1}(x)|_c.\]
 
  Proposition~\ref{character} implies
 ${\rm int}(\Omega_{R(x)-1}(\beta_i)\cap E_{R(x)-1})\subset T_{1,R(x)-2}(B')$.
 In particular,
 $T_{1,R(x)-1}|_{B'}$ is affine
and ${\rm int}(\Omega_{R(x)}(\beta))\subset T_{1,R(x)-1}(B')$.
From the assumption $T_{1,R(x)}(x)\in{\rm int}(X)$, there exists an affine pullback $B$ of ${\rm int}(X)$ by $T_{1,R(x)}$ satisfying $x\in B\subset B'$ and \[|B|_c=\frac{1}{M}|B'|_c=\frac{1}{M}|K_{1}(x)|_c.\] Hence,
the minimum in the proposition does not exceed $R(x)$.
 The reverse inequality is obvious from the property of $B$.
\end{proof}

\if0\subsection{Symbolic coding of nonstationary heterochaos baker maps} \label{Hofbauer-sec}


\begin{prop}\label{symbol-prop}
For any sequence $(a_k)_{k=1}^\infty$ of reals in $(0,\frac{1}{M})$ and any $k\in\mathbb N$, we have \[\overline{\pi_{k}(\Lambda_{k})}=\Sigma_D^+.\]
\end{prop}
\begin{proof}For all $k,n\in\mathbb N$ we have
\[\sigma^n\circ\pi_k|_{\Lambda_k}=\pi_{k+n}\circ f_{k,k+n-1}|_{\Lambda_k}.\]  It is easy to check that 
$\pi_k(\Lambda_k)$ is a shift invariant subset of $D^{\mathbb N}$ that is independent of $(a_k)_{k=1}^\infty$ and $k$. (Not so immediate.)
As shown in \cite[Theorem~1.2]{TY23}, if $a_k=a$ 
for all $k\in\mathbb N$ then the closed shift invariant set
$\overline{\pi_{k}(\Lambda_{k})}$ is the one-sided Dyck shift \cite{Kri74}.
Taking $\Sigma^+$ to be this subshift we obtain the desired statement.\end{proof}
\fi


\subsection{Construction of a partition}\label{ind-exp-sec}
For $\omega\in D^*$ with ${\rm red}(\omega)\neq0$, define
\[W(\omega)=\bigcap_{j=0}^{|\omega|-1}T_{1,j}^{-1}({\rm int}(\Omega_{j+1}(\omega_{j+1}))),\]
where $|\omega|$ denotes the word length of $\omega$
and $\omega=\omega_1\cdots \omega_{|\omega|}$.
It is easy to see that 
$W(\omega)$
is an open rectangle in ${\rm int}(X)$, and
 $T_{1,|\omega|}$ maps $W(\omega)$ affinely onto its image. So, 
$W(\omega)$ is an affine pullback of $T_{1,|\omega|}(W(\omega ))$ by $T_{1,|\omega|}$.

For each integer $n\geq2$,
 define
\[\mathcal{P}_{n}=\{W(\omega)\colon \omega\in D^*,\ |\omega|=n,\ R|_{W(\omega)}=n,\ T_{1,n}(W(\omega))={\rm int}(X)\}.\]
In other words, $\mathcal{P}_{n}$ is
the collection of affine pullbacks 
of ${\rm int}(X)$ by $T_{1,n}$ that are contained in the set $\{R=n\}$.
We set \[\mathcal{P}=\bigcup_{n\geq2}\mathcal{P}_{n}.\] 
Elements of $\mathcal{P}$ are pairwise disjoint open rectangles, and $R$ is constant on each $W\in\mathcal{P}$. This constant is denoted by $R(W)$.
We set \[Y=\bigcup_{W\in\mathcal{P} } W.\]
Propositions~\ref{R-fin} and \ref{exist-cor} together imply $m(X\setminus Y)=0$.
For each $W\in\mathcal{P}$ define a map $F_W\colon W\to X$ by 
$F_W=T_{1,R(W)}|_W$.
Clearly, $F_W$ is affine and $F_W(W)={\rm int }(X)$.

Put 
    \begin{equation}\label{lambda}\lambda=\min\left\{\frac{1}{\overline b},\frac{1}{1-M\underline b}\right\}>1.\end{equation}
By construction, 
for all $W\in\mathcal P$ and $x\in W$ we have
\begin{equation}\label{del1}
\left\|DF_W(x)\left(\begin{smallmatrix}1\\0\end{smallmatrix}\right)\right\|\geq \lambda^{R(W)},\end{equation}
where $D$ denotes the derivative and $\|\cdot\|$ denotes the Euclidean norm.
Moreover, 
for $j=0,\ldots,R(W)-1$ we have 
\begin{equation}\label{del2}\left\|DT_{1,j}(x) \left(\begin{smallmatrix}1\\0\end{smallmatrix}\right)\right\|\leq \lambda^{j-R(W)}\left\|DF_{W}(x) \left(\begin{smallmatrix}1\\0\end{smallmatrix}\right)\right\|,\end{equation} and 
 \begin{equation}\label{del3}\left\|DT_{1,j}(x) \left(\begin{smallmatrix}0\\1\end{smallmatrix}\right)\right\|\leq \frac{1}{M}\left\|DF_{W}(x)\left(\begin{smallmatrix}0\\1\end{smallmatrix}\right)\right\|.\end{equation}
 Since the heterochaos baker maps preserve the $x_u$- and $x_c$-directions, 
 \eqref{del1} and \eqref{del3} with $j=0$
 together imply that
for all $W\in\mathcal P$ and $x,x'\in W$,
\begin{equation}\label{del4}|F_W(x)-F_W(x')|\geq\min\{\lambda,M\}|x-x'|,\ \min\{\lambda,M\}>1.\end{equation}
 Similarly, \eqref{del2} and \eqref{del3}
  together imply that
for all $W\in\mathcal P$ and $x,x'\in W$,
\begin{equation}\label{del5}
	\max_{0 \leq j \leq R(W)} |T_{1,j}(x)- T_{1,j}(x')|
	\leq  |F_W(x)- F_W(x')|.\end{equation}

\subsection{Exponential tail of inducing time}\label{tail-sec}
Throughout Sections~\ref{graph-sec-s} to \ref{tail-sec} we have been working with the one-sided sequence $(b_k)_{k\geq1}$ of parameters in $(0,\frac{1}{2M})$. Under the admissibility condition \eqref{eq:param_range} on this sequence, we deduce an exponential decay of the Lebesgue measure of the tail
probability of inducing time.
\begin{prop}\label{tail-eq0}
If $(b_k)_{k\geq1}$ is admissible, then 
there exist constants $C>0$, $A_0>0$ depending only on $M, \overline{b}, \underline{b}$
such that for all $n\geq 1$,
\[m(R=n)\leq C\exp(-A_0n).\]
\end{prop}

To prove Proposition~\ref{tail-eq0}, we need to estimate the size of each element of the partition
$\mathcal P$, and  the cardinality of the set of elements of $\mathcal P$ with a given inducing time.
For $a$, $b\in(0,\frac{1}{2M}]$ with 
$a\geq b$, 
define 
\[\sigma(a,b)=-\log\sqrt{a(1-Mb)}.\]

\begin{figure}
\begin{center}
\includegraphics[height=5cm,width=15cm]
{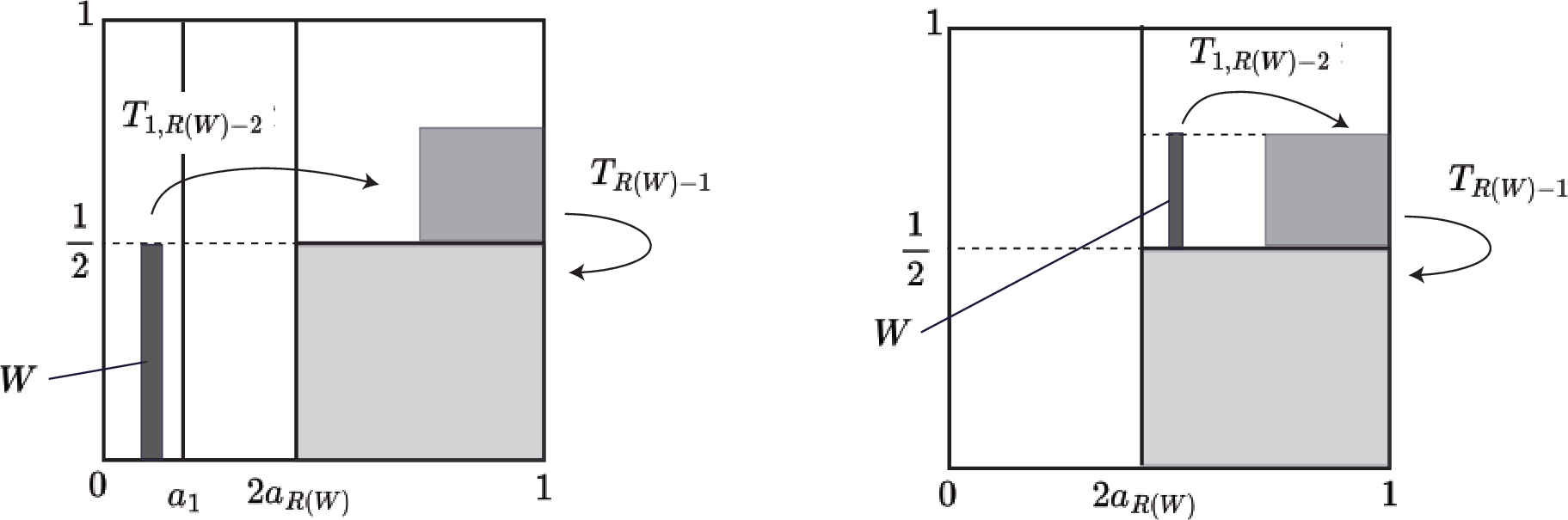}
\caption
{On the proof of Lemma~\ref{area}, the images of each
$W\in\mathcal P$ for $M=2$: $W\subset\Omega_{1}(\alpha)$ and $|W|_c=\frac{1}{M}$ (left); $W\subset\Omega_{1}(\beta)$ and $|W|_c=\frac{1}{M^2}$ (right); $T_{R(W)}\circ T_{R(W)-1}\circ T_{1,R(W)-2}=T_{1,R(W)}$, 
$m(X\setminus T_{1,R(W) }(W ))=0$ in both cases.
  }\label{fig-image}
\end{center}
\end{figure}

\begin{lemma}\label{area}
      For all $W\in\mathcal P$ we have 
    \[
 m(W)\leq
      \exp(-\sigma(\overline b,\underline b)(R(W)-2)).\]
\end{lemma}
\begin{proof}
Recall that $R(W)\geq2$.
For $\gamma=\alpha,\beta$ put
\[R_{W}(\gamma)=\#\left\{0\leq j\leq R(W)-1\colon  T_{1,j}(W)\subset\Omega_{j+1}(\gamma)\right\}.\]
Clearly we have $R_W(\alpha)+R_W(\beta)=R(W).$
 Proposition~\ref{exist-cor} implies the following:
 \begin{itemize}
\item[(i)] If $W\subset\Omega_{1}(\alpha)$, then 
 $|W|_c=\frac{1}{M}$ and
  $\phi_{1,R(W) }=1$ on $W$ (see \textsc{Figure}~\ref{fig-image} left).
  In particular, $R(W)$ is odd and
   \[R_W(\alpha)=\frac{1}{2}(R(W)-1)\ \text{ and }\ R_W(\beta)=\frac{1}{2}(R(W)+1).\]

\item[(ii)] If $W\subset\Omega_{1}(\beta)$, then 
  $|W|_c=\frac{1}{M^2}$ and
   $\phi_{1,R(W) }=2$ 
   on $W$ (see \textsc{Figure}~\ref{fig-image} right).
   In particular, $R(W)$ is even and
      \[R_W(\alpha)=\frac{1}{2}(R(W)-2)\ \text{ and }\ R_W(\beta)=\frac{1}{2}(R(W)+2).\] 
      \end{itemize}

      Under the iteration of each heterochaos baker map $f_a$, the $x_u$-direction is expanding by factor $\frac{1}{a}$
      or $\frac{1}{1-Ma}$ according as on $\Omega_a(\alpha)$ or on $\Omega_a(\beta)$.
Hence we get
\[\begin{split}m(W)=|W|_u
|W|_c\leq|W|_u&\leq \overline b^{R_W(\alpha)}(1-M\underline b)^{R_W(\beta)}\\&\leq\overline b^{\frac{1}{2}(R(W)-2)}(1-M\underline b)^{\frac{1}{2}(R(W)+1)}\\&\leq 
\exp(-\sigma(\overline b,\underline b)(R(W)-2)),\end{split}\]
 as required.
 \end{proof}

The next lemma motivates the admissibility condition \eqref{eq:param_range}. 
Note that
$\frac{4Ma-1}{4M^2a}\leq a$
for all $a\in(0,\frac{1}{2M}]$,
and the inequality is strict only if $a\neq\frac{1}{2M}$.

\begin{lemma}\label{calculus}If $a\in(0,\frac{1}{2M})$ and $b\in(\max\{0,\frac{4Ma-1}{4M^2a}\},a]$, then 
\[\sqrt{4M}e^{-\sigma(a,b)}<1.\] \end{lemma}
\begin{proof}The condition $4M^2ab>4Ma-1$ is equivalent to 
$4Ma(1-Mb)<1$,
which is equivalent to the desired inequality.\end{proof}

\begin{proof}[Proof of Proposition~\ref{tail-eq0}]
We set
\[A_0'=-\log(\sqrt{4M}\exp\left(-\sigma(\overline b,\underline b)\right) ).\]
Since $(b_k)_{k\geq1}$ is admissible, 
Lemma~\ref{calculus} gives $A_0'>0$.
Let $n\geq3$. From Proposition~\ref{character}, for each $W\in\mathcal P$ 
    with $R(W)=n+1$ there exists a unique
    element $[v_1]\cdots [v_{n}]$ of the set $P^*(n)$ in Lemma~\ref{path-number} such that   
     $T_{1,j}(W)\subset v_j$ for $1\leq j\leq n$. 
Combining Lemmas~\ref{path-number} and 
\ref{area} we have
  \[
  \begin{split}m(R=n+1)&=\sum_{\substack{W\in\mathcal P\\ R(W)=n+1} }m(W)\leq\#P^*(n)\exp\left(-\sigma(\overline b,\underline b) (n-1)\right)\\
  &\leq \frac{2M}{n}\begin{pmatrix}n+1\\
  \lfloor\frac{n+2}{2}\rfloor\end{pmatrix}M^{\lfloor\frac{n}{2}\rfloor}
 \exp\left(-\sigma(\overline b,\underline b) (n-1)\right)\\&\leq n^{-\frac{3}{2}}\exp(-A_0'n),\end{split}\]
provided $n\geq n(M)$ where $n(M)$ is a sufficiently large integer depending only on $M$.
    To deduce the last inequality we have evaluated the binomial coefficient using Stirling's formula for factorials.
    Further, setting $A_0=A_0'/2$ 
    and $C=m(R\leq n(M))\exp(n(M)A_0)$
    we obtain the desired inequality for all $n\geq1$.
  \end{proof}

\subsection{Proof of Proposition~\ref{induce-prop}}\label{pf-sec}
Let $(a_k)_{k\in\mathbb Z}$ be a two-sided sequence of parameters in $(0,\frac{1}{2M})$
satisfying \eqref{eq:param_range-z}. Then,
for each $k\in\mathbb Z$  the one-sided sequence $(b_j)_{j\geq1}$ given by 
$b_j=a_{j+k-1}$ for $j\geq1$
is admissible. We apply the construction in Section~\ref{ind-exp-sec} to the corresponding sequence
 $(T_{b_j})_{j\geq 1}$ of heterochaos baker maps and set
  $\mathcal{P}_k=\mathcal{P}$,
$R_k=R$.

Regarding $(\mathcal P_k,R_k)_{k\in\mathbb Z}$,
 items (a), (b) are obvious from the construction.
Item (c) with $\lambda=\min\left\{\frac{1}{\overline a},\frac{1}{1-M\underline a}\right\}$ follows from \eqref{lambda} and \eqref{del4}.
 Item (d) follows \eqref{del5}.
Item (e) follows from
 Lemma~\ref{m-23} and the admissibility of $(b_j)_{j\geq1}$. Item
(f) follows from Proposition~\ref{tail-eq0}. 
The proof of Proposition~\ref{induce-prop} is complete. \qed

\begin{remark}\label{gen-rem}
Proposition~\ref{induce-prop} can be slightly generalized.
Instead of the two-sided sequence of parameters in the mostly expanding center regime $(0,\frac{1}{2M})$, one can consider any sequence $(a_k)_{k\in\mathbb Z}$ of parameters in the whole parameter space $(0,\frac{1}{M})$ for which there exists a subsequence $(a_{k(n)})_{n\in\mathbb Z}$ satisfying 
$a_k\in [\inf_{n\in\mathbb Z}a_{k(n)},\sup_{n\in\mathbb Z}a_{k(n)}]$ for all $k\leq0$, 
\[\max\left\{0,\frac{4M\sup_{n\in\mathbb Z}a_{k(n)}-1}{4M^2\sup_{n\in\mathbb Z}a_{k(n)}}\right\}
 <
 \inf_{n\in\mathbb Z}a_{k(n)}\leq \sup_{n\in\mathbb Z}a_{k(n)}<\frac{1}{2M},\]
 and 
\[\lim_{N\to\infty}\frac{1}{N}\#\left\{1\leq k\leq N\colon a_k\notin[\inf_{n\in\mathbb Z}a_{k(n)},\sup_{n\in\mathbb Z}a_{k(n)}]\right\}=0.\]
In other words, the zero-density appearance of mostly contracting or neutral center parameters is allowed. The main results presented in Section~\ref{sec:results} can be slightly generalized accordingly.
\end{remark}

\section{Stretched exponential loss of memory for nonstationary 
nonuniformly expanding maps}\label{sec:memory_loss}

In this section we consider an abstract class of nonstationary nonuniformly expanding maps admitting a suitable Gibbs--Markov structure. By Proposition~\ref{induce-prop}, this setting applies in particular to nonstationary compositions of heterochaos baker maps in the parameter range \eqref{eq:param_range-z}. We begin by describing the abstract setting, which is similar to that of \cite{KL21}, except that inducing for the nonstationary dynamics is based on general stopping times rather than first return times, in accordance with Proposition~\ref{induce-prop}. By suitably extending the dynamics and adapting the coupling method of \cite{KKM19,KL21}, we derive a stretched exponential rate of memory loss. In Section~\ref{sec:ml_hcb}, we apply this result to the heterochaos baker maps to prove the exponential memory loss stated in Theorem~\ref{thm:main-ml}.

\subsection{Nonstationary nonuniformly expanding maps}\label{sec:nnue}
Let $(X,d)$ be a metric space with
	\[
	\mathfrak{d} = \diam(X) < \infty.
	\]
	We endow $X$ with the Borel sigma-algebra $\mathcal{B}$.
	Let $m$ be a probability measure on $X$.
	We consider a two-sided sequence $(T_k)_{k \in \bZ}$ of 
	measurable 
	transformations $T_k \colon X \to X$ and denote 
	\begin{align*}
	T_{ k, \ell } = \begin{cases}
		T_\ell \circ \cdots \circ T_k, &\text{if $k \le \ell$,} \\
		\text{id}_X, &\text{if $k > \ell$.}
	\end{cases}
	\end{align*}	
	
	Let $Y \subset X$ be a measurable subset with $m(Y) = 1$.
	For a map $\rho \colon Y \to [0, \infty)$, we denote by $| \rho |_{\LL}$ 
	the Lipschitz seminorm of the logarithm of $\rho$:
	\[
	|\rho |_{\LL} = \sup_{\substack{y,y'\in Y\\ y \neq y'}
    } \frac{| \log \rho(y) - \log \rho(y')  |}{ d(y,y') },
	\]
	where 
    $\log 0 = - \infty$ and $\log 0 - \log 0 = 0$ by convention. In the sequel, for 
	a nonnegative measure $\mu$ on $Y$ that is absolutely continuous with respect to $m$, we often
	write $|\mu|_{\LL}$ for $|d \mu / d m|_{\LL}$, 
	with the convention that a density of $\mu$ has been fixed.

    The following assumptions provide an abstract formulation of the
Gibbs--Markov structure established in Proposition~\ref{induce-prop}.
	
    \smallskip

	
	\noindent\textbf{Assumptions.} 
	There exist a sequence 
	$(\cP_k, R_k)_{k \in \bZ}$ and real numbers $\lambda > 1$, 
	$u \in (0,1]$,
	$A, A' > 0$, $K > 0$, 
	such that the following hold for each $k \in \bZ$:
	
    \begin{itemize}
    \item[(A1)]$\cP_k$ is a 
    countable measurable partition of 
    a subset $X_k = \bigcup_{W \in \cP_k} W$ of $Y$ such that
    $m(W) > 0$ for all $W \in \cP_k$ and $m(Y\setminus X_k)=0$. \smallskip 
    \item[(A2)] $R_k \colon X \to [1, \infty]$ is a function that is constant on
    each $W \in \cP_k$ with value $R_k(W) \in \{ 1, 2, \ldots \}$. \smallskip 
    \item[(A3)] For every $W \in \cP_k$, the map $F_W = T_{k, k + R_k(W) - 1} |_W \colon W \to Y$ is a 
    measurable 
    bijection with a measurable inverse, and $T_{k, k + j -1}(W) \in \cB$ whenever $0 \le j < R_k(W)$.
    Further,
    for all $x,x' \in W$,
    \[
    d(F_W(x), F_W(x')) \geq \lambda d(x,x'),
    \]
    and $F_W$ is nonsingular with log-Lipschitz Jacobian with respect to $m$:
    \[
    \zeta = \frac{d (F_W)_*( m |_W )}{dm}
    \quad \text{satisfies} \quad
    |\zeta|_{\LL} \leq K
    .
    \]
    \item[(A4)] For all $x,x' \in W$, with $W \in \cP_k$ and $F_W = T_{k, k + R_k(W) -  1} |_W$ as in (A3),
    \[
    \max_{0 \leq j \leq R_k(W)} d(T_{k,k+j-1}(x), T_{k,k+j-1}(x'))
    \leq K d(F_W(x), F_W(x'))
    .
    \]
    \item[(A5)] There exist $\delta_\# > 0$, an integer $N \ge 1$ 
    and coprime integers $n_1,\ldots, n_N \ge 1$, all independent of $k$, such that
    for all $1 \le j \le N$, 
    \[
    m(  R_k = n_j ) \ge \delta_\#.
    \]
    \item[(A6)] For all $n \ge 1$,
    \[
    h^k(n):=
    m( R_k \ge n ) \le A \exp( - A'  n^u ).
    \]
    \end{itemize}
We do not require $R_k$ to be a first return time to a proper subset of $X$.

    \subsection{Extended dynamics}\label{sec:ext}
	Let $(T_k)_{k \in \bZ}$ be a sequence of maps on $X$ satisfying 
	(A1)-(A6) in Section~\ref{sec:nnue}.
	For each $k \in \bZ$, we define the state space 
	$\bar{X}_k$ as a disjoint union (direct sum)
	\[
	\bar{X}_k = \bar{Y} \sqcup \bigsqcup_{i \ge 1} \bigsqcup_{\substack{W \in \cP_{k - i} \\ R_{k - i}(W) \ge i + 1} } \bar{W}_{k, i},
	\]
	where $\bar{Y} = Y \times \{0\}$ and $\bar{W}_{k,i}$ is a disjoint copy of
	\[
	\{  ( T_{k-i, k-1} (x), R_{k-i} (W) - i) \colon x \in W  \}.
	\]
    For convenience, we use the same notation 
    $(x, r)$
    for a point in
    $\bar{W}_{k,i}$ and its corresponding point in
    $X \times \mathbb Z_+$, as no ambiguity can arise.
	Informally, we consider $\bar{Y}$ as the \say{base} of a tower 
	and $\bar{W}_{k, i}$ as an excursion from the base. 
    Starting from $\bar{W}_{k, i}$, under the 
    time-dependent dynamics 
	$T_k, T_{k+1}, \ldots$, the next return to $\bar{Y}$
    occurs exactly after $R_{k-i}(W) - i$ iterations.
	
	On $\bar{X}_k$, we define a metric $\bar{d}$ as follows:
	\begin{align*}
		\bar{d}(  (x, j ) , (x', j') ) = \begin{cases}
			d(x, x'), &\text{if $(x, j), (x', j') \in \bar{W}_{k,i}$ or $(x,j), (x', j') \in \bar{Y}$}, \\
			\mathfrak{d} + 1, &\text{otherwise.}
		\end{cases}
	\end{align*}
	We equip $\bar{X}_k$ with its Borel sigma-algebra $\bar{\mathcal{B}}_k$. Then the trace 
	sigma-algebras $\bar{\mathcal{B}}_k |_{ \bar{W}_{k, i} }$ and $\bar{ \cB }_k |_{ \bar{Y} }$ are given by 
	\begin{align*}
	\bar{\mathcal{B}}_k |_{ \bar{W}_{k, i} } &= \{  \bar{W}_{k,i}  \cap ( B \times \{   R_{k-i}(W)  - i  \} ) \colon 
	B \in \mathcal{B}  \}, \\ 
	\bar{ \cB }_k |_{ \bar{Y} } &= \{ B \times \{0\} \colon B \in \cB |_Y  \}.
	\end{align*}
	We extend the measure $m$ to a probability measure $\bar{m}$ 	
	on $\bar{Y}$ by 
	\[\bar m( B \times \{ 0 \}  ) = m(B) \quad \text{for $B \in \cB |_Y$.}\]
	For any $k \in \bZ$, we regard $\bar{m}$ as a measure on $\bar{X}_k$ supported on $\bar{Y}$.
	We also extend the definition $| \cdot |_{\LL}$ to functions on $\bar{Y}$ in the obvious way, and write 
	$| \mu |_{\LL} =  |d \mu / d \bar{m} |_{\LL}$ for a nonnegative measure $\mu$ on $\bar{Y}$.
	
	Denoting $\bar{W}_0 = W \times \{0\}$ for $W\in \cP_k$, we extend the partition $\cP_k$ by 
	\begin{align}\label{eq:part_bar}
	\bar{\cP}_k = \{  \bar{W}_0 \colon   W \in \cP_k  \} \cup \{  \bar{W}_{k, i} \colon i \ge 1, \: W \in \cP_{k - i}, \: R_{k-i}(W) \ge i + 1   \}.
	\end{align}
	This is a measurable partition of $\bar X_k$ modulo a subset of $\bar Y$ with 
	$\bar m$-measure zero. 
	
	Finally, we extend the dynamics by defining $\bar{T}_k : \bar{X}_k \to \bar{X}_{k+1}$ as follows:
	\begin{align}\label{eq:ext_dyn}
		\bar{T}_k(x,j) = \begin{cases}
			( T_k(x), j - 1 ), & j \ge 1, \\
			( T_k(x), R_{k}(x) - 1  ), & j = 0, \: x \in X_k.
		\end{cases}
	\end{align}
    Thus, if $(x,j)\in\bar{W}_{k,i}$ with $j\ge 2$, then $\bar{T}_k(x,j)$ belongs to $\bar{W}_{k+1,i+1}$, with the remaining return time decreased by one. If $j=1$, then $\bar{T}_k(x,j) \in \bar{Y}$, so that the trajectory returns to the base. On the other hand, if $(x,0) \in \bar{Y}$ with $x \in W$ for some $W \in \cP_k$, then $\bar{T}_k(x,0)$ either belongs to $\bar{W}_{k+1,1}$ if $R_k(W) \ge 2$, or belongs to $\bar{Y}$ if $R_k(W)=1$.

	If $(x,0) \in \bar N_k := \bar Y \setminus ( X_k \times \{ 0 \} )$ we set $\bar{T}_k(x,0) = (y_*,0)$ where $y_* \in Y$ is a fixed reference point. Since we are assuming that $T_{k, k + i - 1} ( W ) \in \cB$ for 
	$W \in \cP_k$ and $0 \le i \le R_k(W)$, it is straightforward to verify that $\bar{T}_k$ is 
    a measurable map between
	$\bar{ \mathcal{B} }_k$ and $\bar{ \mathcal{B} }_{k+1}$.
	
	For $(x,j) \in \bar{X}_k$, we define 
	$$
	\tau_k( x, j ) = \inf \{  n \ge 1 \colon  \bar{T}_{ k, k + n - 1 }(x, j) \in  \bar{Y}  \}.
	$$
	Note that, for $(x,j) \in \bar X_k \setminus \bar N_k$, we have 
	\begin{align*}
		\tau_k( x,j ) = \begin{cases}
			R_k(x), &\text{if $j = 0$,} \\
			j, &\text{if $j > 0$.}
		\end{cases}
	\end{align*}
    In particular, the first return time $\tau_k$ coincides with $R_k$ on the base 
    $\bar{Y}$.
	Let 
	\[
	\bar{h}^k(n) = \bar{m}( \tau_k \ge n ) = m( R_k \ge n ).
	\]
	Then, (A6) implies that 
	\begin{align}\label{eq:rel_hk_rk}
		\bar{h}^k(n) \le A \exp( - A' n^u  ).
	\end{align}

	\subsection{Regular measures}
	For $\bar{W} \in \bar \cP_k$, set 
	\[
	\bar{F}_{ \bar{W} } = \bar{T}_{k,k+\tau_k( \bar{W} )-1}|_{ \bar W }.
	\]	
    From (A3) it follows that
    $\bar{F}_{ \bar{W} } \colon (x,j) \in \bar{W} \mapsto \bar{F}_{ \bar{W} }(x,j)\in\bar{Y}$ is a measurable bijection onto $\bar{Y}$ with measurable inverse.
	
	\begin{prop}
		\label{prop:K}
        Let 
        \begin{align}\label{eq:k1k2}
		K_1 = K + \lambda^{-1} K_2 \quad \text{and} \quad K_2 > (1 - \lambda^{-1})^{-1} K.
		\end{align}
        Then $K < K_1 < K_2$ and, for each $k \in \bZ$ and
		for each nonnegative measure $\mu$ on $\bar Y$
		with $|\mu|_{\LL} \leq K_2$,
		\[
		\bigl| (  \bar{F}_{\bar W_0 } )_* (\mu|_{\bar W_0 }) \bigr|_{\LL} \leq K_1
		,
		\]
		whenever $\bar W_0 = W \times \{ 0 \} \in \bar \cP_k$.
	\end{prop}
	
	\begin{proof} The proof is identical to \cite[Proposition~3.1]{KKM19}, but we provide 
	the details below for the reader's convenience.
		For $\bar W_0 = W \times \{ 0 \} \in \bar\cP_k$, we have 
		\begin{align}\label{eq:pw_density}
			v := \frac{ d ( \bar{F}_{ \bar W_0 })_* ( \mu |_{\bar W_0 }) }{d\bar{m} } = \rho( \bar{F}_{\bar W_0}^{-1}  ) \cdot  \bar{\zeta},
		\end{align}
		where
		\begin{align*}
			\rho = \frac{ d \mu  }{d \bar m } \quad \text{and} \quad 
			\bar{\zeta}(y,0) = \frac{ d( \bar{F}_{ \bar W_0 })_*( \bar{m} |_{\bar W_0 } ) }{ d \bar{m}  } (y,0)
			= \frac{ d(F_{W})_*(m|_{ W }) }{dm}(y).
		\end{align*} 
		 Thus, for $(y,0), (y', 0) \in \bar Y$,
		\begin{align*}
			&|\log v(y,0) - \log v(y',0) | \\
			&\le | \log \rho( \bar F_{\bar W_0}^{-1}(y,0) ) - \log \rho( \bar F_{\bar W_0 }^{-1}(y',0) ) | + |  \log  \bar \zeta( y,0 ) - 
			\log \bar  \zeta( y',0 ) | \\
			&\le (  |\mu|_{\LL} \lambda^{-1}  + |\zeta|_{\LL}  ) d( y, y' ) \\
            &= (  |\mu|_{\LL} \lambda^{-1}  + |\zeta|_{\LL}  ) \bar{d}( (y, 0), (y',0) ).
		\end{align*}
		where (A3) was used in the second inequality. The desired inequality follows
        from $|\mu|_{\LL} \leq K_2$ and \eqref{eq:k1k2}.
	\end{proof}

    The following definition is a minor variation of \cite[Definition 3.5]{KL21}.
	
	\begin{definition} \label{def:reg}
		Fix $K_1 < K_2$ as in Proposition~\ref{prop:K}.
		For $k \in \bZ$ and $\ell \ge 1$, define
		$$
		B_{k,\ell}
		= \{   (x,j) \in \bar{X}_k \colon \tau_k(x,j) = \ell  \}.
		$$
		A nonnegative measure $\mu$ on $\bar{X}_k$ is called {\it regular}
		if 
		$
		\mu( \bar{N}_k ) = 0
		$
		and
		for every $\ell \geq 1$,
		\begin{align}\label{eq:regular}
			\bigl| ( \bar{T}_{k, k + \ell - 1} )_* (\mu|_{ B_{k, \ell} }) \bigr|_{\LL}
			\leq K_1.
		\end{align}
		Given  $r \colon \{1,2, \ldots\} \to [0,\infty)$,
		we say that $\mu$ has {\it tail bound} $r$ 
		if for all $n \geq 1$, 
		\begin{align}\label{eq:tail_bound}
			\mu \bigl(    \tau_k \geq n   \bigr)
			\leq r( n )
			.
		\end{align}
	\end{definition}

Throughout the rest of the paper, we only consider finite measures.
	As in \cite[Proposition~4.1]{KL25}, regular measures satisfy the following 
    basic 
    properties. 
    
	\begin{prop}\label{prop:regular} Let $k \in \bZ$.
		\begin{itemize}
			\item[(a)] If $\{\mu_j\}$ is a finite or countable collection of regular
			measures on $\bar{X}_k$, 
			then $\mu = \sum_{j} \mu_j$ is regular.  \smallskip
			\item[(b)] The measure $\bar m$ (as a measure on $\bar{X}_k$) is regular with tail bound $r = \bar h^k$ 
			and every measure $\mu$ on $\bar Y \subset \bar{X}_k$ with $|\mu|_{\LL} \leq K_2$ is a regular measure on $\bar{X}_k$ with tail bound
			$r  \le 
            \mu( \bar Y ) e^{ \mathfrak{d}  K_2} \bar h^k$. \smallskip
			\item[(c)] 	If $\mu$ is a regular measure on $\bar{X}_k$,  then 
			$(\bar{T}_{k, k + n -1})_*\mu$ 
			and $( (\bar{T}_{k, k + n -1})_*\mu)_{  \bar{X}_{k+n} \setminus \bar Y }$
			are both regular measures on $\bar{X}_{k + n}$.
			Moreover, for all $n \ge 1$,
			\[
			\Bigl| \bigl( (\bar{T}_{k, k +  n - 1})_* \mu \bigr) |_{ \bar Y }  \Bigr|_{\LL}
			\leq K_1
			.
			\]
		\end{itemize}
	\end{prop}
	
	\begin{proof} 
		
		\textbf{(a)} The claim follows from the fact that if $\nu = \sum_{j} \nu_j$, where
		$\{\nu_j\}$ is a countable 
        collection of nonnegative measures on $\bar Y$,
		then
		\begin{align}\label{eq:sum_ll}
			| \nu |_{\LL} \le \sup_{j} | \nu_j|_{\LL}.
		\end{align}

        \smallskip 
		
		\noindent \textbf{(b)} Regularity of $\bar m $ and $\mu$ follow 
		by combining Proposition~\ref{prop:K} and part (a), and the tail bound 
		of $\mu$ is a consequence of the standard estimate 
        \[
        \frac{d\mu}{d \bar{m} } \le 
        \inf_{ y \in \bar{Y} } \frac{d\mu}{d \bar{m} } (y)  e^{ \mathfrak{d}  K_2} \bar h^k
        \le \mu( \bar Y ) e^{ \mathfrak{d}  K_2} \bar h^k.
        \]
        
        \smallskip
        
		\noindent \textbf{(c)} It suffices to prove the case $n = 1$. Let $\mu$ be a regular measure on $\bar{X}_k$.
		Since $\mu(\bar N_k) = 0$, we have 
		$$
		( (\bar{T}_k )_* \mu ) |_{ \bar Y } = ( \bar{T}_k)_*( \mu |_{  B_{k,1}  } ).
		$$
		Therefore, by the definition of regularity,
		$$
		| ( (\bar{T}_k )_* \mu ) |_{ \bar Y }  |_{\LL} \le K_1.
		$$
        In particular, $( (\bar{T}_k )_* \mu ) ( \bar{N}_{k + 1} ) = 0$ because 
        $\bar{m}( \bar{N}_{k + 1} ) = 0$.
        
		Next, we verify that $(\bar{T}_k)_*\mu$ is a regular measure 
		on $\bar{X}_{k+1}$.
		To this end, let $\ell \geq 1$, and denote 
		$\mu' = ((\bar{T}_k)_* \mu) |_{\bar Y}$. We decompose 
		\begin{align*}
			( \bar{T}_{ k + 1, k + 1  + \ell - 1  } )_* \bigl(  ( (\bar{T}_k)_*\mu ) |_{B_{k+1, \ell}}  \bigr) = \mu_1 + \mu_2,
		\end{align*}
		where 
		\begin{align*}
			\mu_1 = ( \bar{T}_{k+1, k+\ell} )_*[ \mu' |_{ B_{k+1, \ell } } ] \quad \text{and} \quad 
			\mu_2 = ( \bar{T}_{k, k + \ell } )_*[  \mu |_{ B_{k, \ell + 1} } ].
		\end{align*}
		Since $|\mu'|_{\LL} \le K_1$, it follows from part (b) that $\mu'$ is a regular measure on $\bar{X}_{k + 1}$.
		Therefore, 
		$|\mu_1|_{\LL} \le K_1$. Moreover, 
		by the regularity of 
		$\mu$ we have $|\mu_2|_{\LL} \le K_1$. From \eqref{eq:sum_ll} we now obtain
		\[
		|( \bar{T}_{ k + 1, k + 1  + \ell - 1  } )_* \bigl(  ( (\bar{T}_k)_*\mu ) |_{B_{k+1, \ell}}  \bigr) |_{\LL} \le K_1.
		\]
		Finally, note that the regularity of $( (\bar{T}_{k})_*\mu)_{  \bar{X}_{k + 1} \setminus \bar Y }$ follows 
		from $|\mu_2|_{\LL} \le K_1$.
	\end{proof}

    \subsection{Loss of memory}

    Let $( \bar{T}_k )_{k \in \bZ}$ be a sequence of maps 
    as in Section~\ref{sec:ext}. The proof of the following result is similar to 
    \cite[Theorem~3.8]{KL21}. We provide the details of the proof in
    Appendix~\ref{sec:app_a}.

	
	\begin{thm}\label{thm:decdec}
        Let $k \in \bZ$.
		Suppose that $\mu$ is a regular 
        probability 
        measure on $\bar{X}_k$ with tail bound $r$. Then there exists a decomposition 
		\[
		\mu = \sum_{n=1}^\infty \alpha_{n} \mu_{n},
		\]
		where $\mu_{n}$ are probability measures on $\bar{X}_k$
		such that $( \bar{T}_{k, k + n - 1 })_* \mu_n = \bar m$ for each $n \geq 1$, and $\alpha_n$ are nonnegative 
        constants
		with $\sum_{n \geq 1} \alpha_n = 1$. The sequence $(\alpha_n)_{n\geq1}$
		is fully determined by $K_1$, $K_2$, the system constants ($\mathfrak{d}$, $K$, $\lambda$, $\delta_\#$, $N$, $n_1,\ldots, n_N$), and
		the functions $r$ and $(h^j)_{j \ge k}$ in (A6).
		In particular, $(\alpha_n)_{n\geq1}$ does not depend on $\mu$
		in any other way. Moreover, assume that 
		\begin{align}\label{eq:rn}
		r(n) \le C_r \exp( - C'_r n^u ), \quad n \ge 1.
		\end{align}
		Then, 
		\[
		\sum_{j \geq n} \alpha_j
		\leq C \exp(-C' n^u ), \quad n \ge 1.
		\]
		The constants $C, C'$ depend only on 
		$u, A, A', C_r, C_r', K_1, K_2$, and the system constants.
	\end{thm}

    From Theorem~\ref{thm:decdec} we deduce a stretched exponential rate of memory loss
    for the original maps
	$T_1,T_2,\ldots$. To this end, 
	given a measure $\mu$ on $Y$, we denote by $\bar \mu$ the measure on
	$\bar Y$ defined by
	\begin{align}\label{eq:mu_bar}
	\bar \mu (B \times \{ 0 \}) = \mu (B), \quad B \in \cB |_Y.
	\end{align}
	
	\begin{cor}\label{cor:ml_orig_maps} Assume that $\mu_1$ and $\mu_2$ are probability measures on $Y$ 
	such that $\bar \mu_1$ and $\bar \mu_2$ are regular measures on 
	$\bar X_1$. Moreover, assume that, for 
	$i = 1,2$ and all $n \ge 1$,
	\[
	\mu_i(  R_1 \ge n ) \le C_r \exp( - C_r' n^u ).
	\]
	Then, for all $n \ge 1$,
		\[
		| (T_{1,n})_*\mu_1 - (T_{1,n})_* \mu_2 | \le C \exp(  - C'n^u  ).
		\]
        The constants $C, C' > 0$ 
		depend only on $A, A', C_r, C_r', K_1, K_2$, and the system constants ($\mathfrak{d}$, $K$, $\lambda$, $\delta_\#$, $N$, $n_1,\ldots, n_N$).
	\end{cor}
	
	\begin{proof} First, Theorem~\ref{thm:decdec} implies 
    
		\begin{align*}
		\sup_{  B \in  \bar \cB_{n + 1}  }	| (\bar{T}_{1,n})_* \bar \mu_1(B) -  (\bar{T}_{1,n})_* \bar \mu_2(B)  | \le 2 \sum_{j > n} \alpha_j  \le C \exp(-C'n^u).
		\end{align*}
		For any $B \in \cB$ we have that $( B \times \bZ_+ ) \cap \bar{X}_{n+1} \in \bar{\cB}_{n + 1}$ and 
		\[
		\bar{T}_{1,n}^{-1}[ ( B \times \bZ_+ ) \cap \bar{X}_{n+1} ] \cap \bar Y
		= ( T_{1,n}^{-1} (B) \times \{0\} ) \cap \bar Y.
		\]
		Hence,
		\begin{align}\label{eq:ext_ml}
		\sup_{ B \in \cB } |   (T_{1,n})_* \mu_1 ( B)   - (T_{1,n})_* \mu_2 ( B)| \le 
		\sup_{  B \in  \bar \cB_{n + 1}  }	| (\bar{T}_{1,n})_* \bar \mu_1(B) -  (\bar{T}_{1,n})_* \bar \mu_2(B)  |
		\end{align}
		and the desired estimate follows.
	\end{proof}

\subsection{Proof of Theorem~\ref{thm:main-ml}}\label{sec:ml_hcb}

Let $(a_k)_{k \ge 1}$ be a sequence of parameters in $(0,\frac{1}{2M})$ that satisfies \eqref{eq:param_range}. Recall that $m$ is the Lebesgue measure on $X = [0,1]^2$, 
which is equipped with the Euclidean metric $d(x,y) = |x - y|$.
Let $Y = (0,1)^2$.
Recall also that $T_k = f_{a_k}$ whenever $k \ge 1$. For convenience we set 
$T_k = f_{a_1}$ for $k\le 0$. 
By Proposition~\ref{induce-prop}, the two-sided sequence $(T_k)_{k \in \bZ}$ satisfies assumptions (A1)-(A6) in Section~\ref{sec:nnue}. 
In particular:
\begin{itemize}
    \item (A1) holds with $Y = (0,1)^2$;
    \item in (A3) the induced map $F_W = T_{k,k+R_k(W) - 1}|_W$ for $W \in \cP_k$
    is an affine map onto $Y$;
    \item (A4) holds with $K=1$;
    \item (A5) holds with $N=2$, $n_1 = 2$, and $n_2  = 3$;
    \item (A6) holds with $u = 1$.
\end{itemize}

We define the extended map $\bar T_k$ as in~\eqref{eq:ext_dyn}.
Let $\mu_i$ ($i = 1,2$) be  a probability measure on $Y$ whose density $\rho_i$ is Lipschitz
continuous with coefficient $L$, such that  
$\inf_{y \in Y} \rho_i(y) \ge \underline{\rho}$ for $i = 1,2$ and some $\underline{\rho} > 0$.
Then, for $W \in \cP_k$,  
\[
\frac{ (F_W)_* ( \mu_i |_W  )  }{dm} (x) = \frac{ \rho_i ( F_W^{-1} x) }{  \det DF_W ( F_W^{-1} (x) )  }.
\]
Moreover, $| \mu_i |_{\LL} \le \underline{\rho}^{-1} L$. It follows from 
Proposition~\ref{prop:K} that $\bar \mu_i$ 
is regular on $\bar X_k$ for $i = 1,2$
provided that we choose $K_2 > \underline{\rho}^{-1} L$ in 
Definition~\ref{def:reg}. This is possible, since we can take $K_2$
in Proposition~\ref{prop:K} to be arbitrarily large. Fixing such $K_2$,
Corollary~\ref{cor:ml_orig_maps} implies
\begin{align}\label{eq:ml_y}
| (T_{1,n})_*\mu_1 - (T_{1,n})_* \mu_2 | \le C \exp(  - C'n  ),
\end{align}
where $C, C'$ depend only on $L, \underline{\rho}$ and $M, \overline{a}, \underline{a}$. 

Next, let $\mu_i$ ($i=1,2$) be probability measures on $Y$ whose
densities $\rho_i$ are Lipschitz continuous with coefficient $L$, without
assuming that they are bounded away from zero. In this case we write
\[
(T_{1,n})_*\mu_1 - (T_{1,n})_*\mu_2 =  2 (T_{1,n})_*\mu_1' - 2 (T_{1,n})_*\mu_2',
\] 
where 
\[
\mu_i' = \frac{ \mu_i + m }{2}.
\]
Then $\mu_i'$ is a probability measure on $Y$ with density $\rho_i' \ge 1/2$
that is Lipschitz continuous with coefficient $L/2$, 
so that \eqref{eq:ml_y} extends 
to such measures. Further, \eqref{eq:ml_y}  extends to measures whose densities
are H\"older continuous with exponent $\theta$ by changing the metric to $d(x,y) = |x - y|^\theta$. Finally, since 
$m(X \setminus Y) = 0$,  \eqref{eq:ml_y} readily extends to measures $\mu_i$ on 
$X$ with H\"older continuous densities. The proof of Theorem~\ref{thm:main-ml} is complete. \qed

\section{Proof of Theorem~\ref{thm:main-fcb}}\label{sec:fcb_proof}

Let $T_k$ and $\bar{T}_k$ be defined as in Section~\ref{sec:ml_hcb}.
Let $\mu$ be a probability measure on $X$ whose density $\rho$ is Lipschitz
continuous with coefficient $L$, such that  
$\underline{\rho} = \inf_{x \in X} \rho(x) > 0$. Recall that $\bar{\mu}$ is defined 
by \eqref{eq:mu_bar}. As in Section \ref{sec:ml_hcb}, 
we choose $K_2 > \underline{\rho}^{-1} L$ in 
Definition~\ref{def:reg} so that 
$\bar \mu$ 
is regular on $\bar X_k$ for any $k \in \bZ$.


We define 
recursively a sequence $(\chi_j)_{j \ge 1}$ of stopping times by 
\[
\chi_j(y) = \chi_{j-1}(y) + R_{ \chi_{j - 1}(y) + 1  } \circ T_{{1, \chi_{j - 1}(y) }}, \quad y \in Y,
\]
whenever $\chi_{j-1}(y)<\infty$, and set $\chi_j(y)=\infty$ otherwise,
where $\chi_0=0$.
Note that $\chi_1 = R_1$. 

In terms of the extended dynamics $\bar T_j$, $\chi_j$ can be interpreted as the $j$th successive return time to $\bar Y$ in the following sense. Let
\[
\bar \tau_j (x) = \inf \{  \ell > \bar \tau_{j - 1}(x) \colon  \bar T_{1, \ell }(x) \in \bar Y   \}, \quad j \ge 1,
\]
where $\bar\tau_0 = 0$. Then, for $y \in Y$,
\[
\bar\tau_j( y, 0 ) = \chi_j(y).
\]
The above identities hold on a subset of $Y$ with full $m$-measure and we ignore the null set as this does not create any issues.

Next, we define
\[
L(n)(y) = \max (  \{ 0 \} \cup   \{  1 \le \ell \le n  \colon  \chi_{\ell}(y) \le n    \}  ) = \# \{  1 \le  \ell \le n \colon \bar T_{1,\ell}(y,0) \in \bar Y   \}.
\]
Thus $L(n)$ counts the number of returns to $\bar Y$ by time $n$.
Note that, by definition, for all $y\in Y$,
\[
\chi_{_{L(n)(y)  }} (y)\le n <   \chi_{_{L(n)(y) + 1}}(y).
\]

Given 
$n \ge 1$, we consider the (mod $m$) partition of $Y$ into the sets

\begin{align}\label{eq:cylinder}
\begin{split}
&C_n(\ell_1, \ldots, \ell_{p+1}; W_1, \ldots, W_{p+1}) \\
&= \{ \chi_1 = \ell_1, \ldots, \chi_{p + 1} = \ell_{p + 1} \}
\cap \{L(n) = p\}
\cap \bigcap_{j=0}^p T_{1,\ell_j}^{-1}(W_{j+1}) \\
&= 
\bigcap_{j=0}^p
\left\{
R_{\ell_j+1}\circ T_{1,\ell_j}
=
\ell_{j+1}-\ell_j
\right\}
\cap
\bigcap_{j=0}^p T_{1,\ell_j}^{-1}(W_{j+1}),
\end{split}
\end{align}
where 
\begin{itemize}
    \item $\ell_0 = 0$ and
    $1 \le \ell_1 < \cdots < \ell_{p} \le n < \ell_{p + 1}$ if $1 \le p \le n$;
    \item $\ell_1 > n$ if $p = 0$; and
    \item $W_1 \in \cP_1, W_2 \in \cP_{ \ell_1 + 1 }, \ldots, 
    W_{p + 1} \in \cP_{\ell_p + 1}$.
\end{itemize}
In terms of the extended dynamics, this corresponds to partitioning $\bar Y$, into countably many \say{cylinders} 
$C_n(\ell_1, \ldots, \ell_{p+1}; W_1, \ldots, W_{p+1})$
such that on each such cylinder $\ell_1 < \cdots < \ell_{p}$ 
enumerate return times to $\bar Y$ up to time $n$, $\ell_{p + 1}$ is the first return time 
after $n$, and $W_j$ is the unique element of $\cP_{\ell_j + 1}$ containing the state of the 
system after the $j$th return.

\subsection{Proof of the functional correlation bound} 
Recall that
\[
H(x,y) = F( T_{1, i_1}(x), \ldots, T_{1, i_l}(x),  T_{1, i_{l+1}}(y), \ldots, T_{1, i_n } (y) ).\]
We define integers as follows:
\[
i_* = i_l + \lfloor  g / 3  \rfloor , \quad g = i_{l + 1} - i_l, \quad i_\# = i_l  + \lfloor  2 g / 3  \rfloor.\]
Without loss of generality we assume that $g > 3$.  Throughout the proof 
we denote by $C,C'$ constants depending only on $\theta, L, \underline{\rho}, M, \overline{a}, \underline{a}$.
The values of $C, C'$ may change from line to line. We write $x = O(y)$ if 
$|x| \le C |y|$. 

We begin by decomposing
\begin{align*}
	\mathcal{I} &= \int_X H(x,x) \, d\mu (x) - \iint_{ X^2 } H(x,y) \, d\mu (x) \, d \mu (y) \\
	&=  \sum_{0 \le p \le i_*, \: \ell_1 < \cdots < \ell_{p+1}} 
    \sum_{  W_1,  \ldots, 
    W_{p + 1}  }
    \int_{ C_{i_*}( \tilde{ \ell }, \tilde{W} ) } 
	\biggl[ H(x,x) - \int_X H(x,y) \, d \mu (y) \biggl] \, d\mu (x),
\end{align*}
where the sums are taken over all return-time sequences 
and
partition-element sequences as in \eqref{eq:cylinder}, and for brevity
we have written
\[
C_{i_*}( \tilde{ \ell }; \tilde{W} ) = C_{i_*}(\ell_1, \ldots, \ell_{p+1}; W_1, \ldots, W_{p+1}).
\]

Next, we define
\[
Q_{l}^* = L(i_*) - L(i_l). 
\]
Note that $Q_{l}^*$ is constant on the set $C_{i_*}( \tilde{ \ell }, \tilde{W} )$, as all return times up to time $i_*$ are determined. In terms of the extended dynamics $\bar T_{1,k}$,
$Q_{l}^*$ is the number of returns to $\bar Y$ within the time interval $(i_l, i_*]$. We further decompose $\mathcal{I} = \mathcal{I}_1 + \mathcal{I}_2 + \mathcal{I}_3$ into three terms according to the values of $Q_l^*$ and $\ell_{p + 1}$ as follows:
\begin{align*}
	\mathcal{I}_1 &=  \sum_{ \substack{0 \le p \le i_*, \: \ell_1 < \cdots < \ell_{p+1} \\ \ell_{p+1} > i_\# } } 
    \sum_{  W_1,  \ldots, 
    W_{p + 1}  }
    \int_{ C_{i_*}( \tilde{ \ell }, \tilde{W} ) } 
	\biggl[ H(x,x) - \int_X H(x,y) \, d \mu (y) \biggl] \, d\mu (x), \\
	\mathcal{I}_2 &=  \sum_{ \substack{0 \le p \le i_*, \: \ell_1 < \cdots < \ell_{p+1} \\ \ell_{p+1} \le i_\#, \: Q_{l}^*  > t }  } 
    \sum_{  W_1,  \ldots, 
    W_{p + 1}  }
    \int_{ C_{i_*}( \tilde{ \ell }, \tilde{W} ) } 
	\biggl[ H(x,x) - \int_X H(x,y) \, d \mu (y) \biggl] \, d\mu (x), \\
		\mathcal{I}_3 &=  \sum_{ \substack{0 \le p \le i_*, \: \ell_1 < \cdots < \ell_{p+1} \\ \ell_{p+1} \le i_\#, \: Q_{l}^* \le t }  } 
        \sum_{  W_1,  \ldots, 
    W_{p + 1}  }
        \int_{ C_{i_*}( \tilde{ \ell }, \tilde{W} ) } 
	\biggl[ H(x,x) - \int_X H(x,y) \, d \mu (y) \biggl] \, d\mu (x).
\end{align*}
Here $t \ge 0$ is an integer to be determined later. In the remainder of the proof, we show that 
$\mathcal{I}_1$ and $\mathcal{I}_3$ are negligible if $t$ is chosen suitably, and then use 
Theorem~\ref{thm:decdec} to control $\mathcal{I}_2$.

\smallskip 

\noindent\textbf{Bound on $\mathcal{I}_1$}: Observe that 
\begin{align*}
	|\mathcal{I}_1| \le 2 \Vert F \Vert_\infty \mu ( \chi_{  L(i_*) + 1   }   \ge i_\#    )
	\le  2 \Vert F \Vert_\infty \mu (   R_{ \eta (i_*)  + 1  } \circ T_{1, \eta(i_*)} \ge    i_\#  - \eta(i_*)   ),
\end{align*}
where $\eta(i) = \chi_{ L(i) } \le i$.
Using the property that each map $T_k$ preserves $m$, and Proposition \ref{induce-prop}-(f), 
we obtain
\begin{align}\label{eq:estim_I1}
	\begin{split}
		|\mathcal{I}_1| &\le 2 \Vert F \Vert_\infty  
		\Vert \rho \Vert_\infty 
		\sum_{ j = 0 }^{ i_* } m(  \{  \eta(i_* ) = j  \} \cap \{   R_{ j + 1 } \circ T_{1,j}  \ge i_\# - j  \}  ) \\
		&\le 2  \Vert F \Vert_\infty  \Vert \rho \Vert_\infty  \sum_{ j = 0 }^{ i_* } m(     R_{ j + 1 } \circ T_{1,j}  \ge i_\# - j     )  \\
		&=  2 \Vert F \Vert_\infty \Vert \rho \Vert_\infty  \sum_{ j = 0 }^{ i_* } m(     R_{ j + 1 } \ge i_\# - j    ) 
		\le  C \Vert F \Vert_\infty    \exp( - C' g  ).
	\end{split}
\end{align}


\noindent\textbf{Bound on $\mathcal{I}_3$}:  We have 
\[
|\cI_3| \le 2 \Vert F \Vert_\infty \mu (  Q_l^* \le t ).
\]
Since $i_* - i_l = \lfloor g / 3 \rfloor$ and $Q_l^* = L(i_*) - L(i_l)$,
\begin{align*}
	\{  Q_l^* \le t  \} &= \{  Q_l^* \le t  \} \cap  \biggl\{  
	\sum_{j=0}^{ Q_l^* - 1 } (     \chi_{ L(i_l) + j + 1 }  - \chi_{ L(i_l) + j }     ) + (   \chi_{ L(i_*)  + 1 } -  \chi_{ L(i_*) }    ) \ge 
	\lfloor g / 3 \rfloor
	\biggr\} \\
	&\subset \biggl( \{  Q_l^* \le t  \} \cap \bigcup_{j=0}^{ Q_l^* - 1 } \{  \chi_{ L(i_l) + j + 1 }  - \chi_{ L(i_l) + j } \ge  \lfloor g / 3 \rfloor  / (t+1)  \} \biggr) \\
	&\qquad \cup \{   \chi_{ L(i_*)  + 1 } -  \chi_{ L(i_*) }  \ge \lfloor g / 3 \rfloor  / (t + 1)  \}
	\\
	&= 
    \biggl( \{  Q_l^* \le t  \} \cap 
	\bigcup_{j=0}^{Q_l^*-1} \{  R_{ \kappa(i_l,j)  + 1} \circ T_{1, \kappa(i_l,j) } \ge  \lfloor g / 3 \rfloor  / (t+1)  \} \biggr) \\
    &\qquad \cup \{  R_{ \kappa(i_*,0) + 1 } \circ T_{1, \kappa(i_*,0) } \ge \lfloor g / 3 \rfloor  / (t + 1)  \},
\end{align*}
where $\kappa(k,j) = \chi_{ L(k) + j }$. Using the $T_k$-invariance of $m$ and Proposition \ref{induce-prop}-(f), we obtain 
\begin{align*}
	&m(   R_{ \kappa(i_*,0) + 1 } \circ T_{1, \kappa(i_*,0) } \ge \lfloor g/3 \rfloor / (t + 1)  )
    \\
	&= \sum_{ j =  0 }^{i_*} m(  \{  \chi_{ L(i_*) } = j    \}  \cap \{  R_{j + 1} \circ T_{1,j} \ge \lfloor g / 3 \rfloor / (t + 1)    \}  ) \\
	&\le C ( i_* + 1 )\exp(  - C' g / (t + 1)  ),
\end{align*}
and, 
\begin{align*}
	&m \biggl(  \{  Q_l^* \le t  \} \cap  \bigcup_{j=0}^{Q_l^*-1} \{  R_{ \kappa(i_l,j)  + 1} \circ T_{1, \kappa(i_l,j) } \ge  \lfloor g / 3 \rfloor  / (t+1)  \}  \biggr) \\
    &\le 
    m \biggl( \bigcup_{j=0}^{t-1} \bigcup_{r=0}^{i_*} 
    \{ \kappa(i_l,j) = r  \}  \cap  \{  R_{r + 1} \circ T_{1,r} \ge \lfloor g / 3 \rfloor  / (t + 1)    \}
    \biggr) \\
    &\le \sum_{j = 0}^{t-1} \sum_{  r = 0  }^{ i_* } 
	m( \{ \kappa(i_l,j) = r  \}  \cap  \{  R_{r + 1} \circ T_{1,r} \ge \lfloor g / 3 \rfloor  / (t + 1)    \}   ) \\
	&\le C ( i_* + 1 ) t \exp(  - C'  \lfloor g / 3 \rfloor / (t + 1)  ).
\end{align*}
Therefore, 
\begin{align*}
	m(  Q_l^* \le t   ) \le C i_* t \exp(  - C'   \lfloor g / 3 \rfloor  / (t + 1)  ).
\end{align*}
We choose $t = \lceil  \sqrt{g} \rceil$. It follows that 
\begin{align}\label{eq:estim_I3}
	|\mathcal{I}_3|  \le    C  \Vert F \Vert_\infty \Vert \rho \Vert_\infty  (i_l + 1)   \exp( - C' \sqrt{g}  ).
\end{align}

\smallskip 

\noindent\textbf{Bound on $\mathcal{I}_2$}: Let $x,x' \in C_{i_*}( \tilde{ \ell }; \tilde{W} )$, 
where $\ell_{p+1} \le i_{\#}$, and $Q_l^* > t$. 
By the definition of $C_{i_*}(\tilde{\ell};\tilde{W})$, repeated applications of Proposition~\ref{induce-prop}-(c) together with Proposition~\ref{induce-prop}-(d) give, for any $y\in X$,
\begin{align*}
|H(x, y) - H(x', y)| \le \sum_{k=1}^l  \Vert F \Vert_{\theta}  | T_{1, i_k }(x) - T_{1,i_k}(x') |^\theta \le  C  \Vert F \Vert_{\theta}
 l \lambda^{-\theta t}.
\end{align*}
Hence, fixing $c( \tilde{ \ell }; \tilde{W}  ) \in C_{i_*}( \tilde{ \ell }; \tilde{W} )$, we have 
\begin{align*}
\mathcal{I}_2 &=  \sum_{ \substack{0 \le p \le i_*, \: \ell_1 < \cdots < \ell_{p+1} \\ \ell_{p+1} \le i_\#, \: Q_{l}^*  > t }  }  
\sum_{  W_1,  \ldots, 
    W_{p + 1}  }
\int_{ C_{i_*}( \tilde{ \ell }; \tilde{W} )  } 
\biggl[ H( c( \tilde{ \ell }; \tilde{W}  )  ,x) - \int_X H(  c( \tilde{ \ell }; \tilde{W}  )  ,y) \, d \mu (y) \biggl] \, d\mu (x) \\
&\qquad+ O(  \Vert F \Vert_{\theta}  i_l  \lambda^{-\theta t} ).
\end{align*}

Suppose that $\mu(  C_{i_*}(\tilde{\ell};\tilde{W}) ) > 0$. We denote by $\nu$ the probability measure 
on $X$ defined by
\begin{align*}
	 \nu(A) = \frac{ \mu(  C_{i_*}(\tilde{\ell};\tilde{W}) \cap A ) }{ \mu(  C_{i_*}(\tilde{\ell};\tilde{W}) ) },
\end{align*}
and let $\tilde{H}(  c( \tilde{ \ell }; \tilde{W}  ) , x  )$ be the function that satisfies 
\[
\tilde{H}(  c( \tilde{ \ell }; \tilde{W}  ) ,  T_{1, i_{l + 1} } x  ) = H(  c( \tilde{ \ell }; \tilde{W}  ) ,  x  ).
\]
Then, we have
\begin{align*}
	\bar\cI &:= \int_{ C_{i_*}(\tilde{\ell};\tilde{W}) } \biggl[
	H(  c( \tilde{ \ell }; \tilde{W}  ) , x )
	-  \int_X   H(  c( \tilde{ \ell }; \tilde{W}  ) , y ) \, d\mu (y)  \biggr] \, d\mu (x) \\
	&=  \mu (   C_{i_*}(\tilde{\ell};\tilde{W})  ) \int_X \biggl[
	\tilde{H}( c( \tilde{ \ell }; \tilde{W}  ), \cdot )
	-  \int_X   H( c( \tilde{ \ell }; \tilde{W}  ) , y ) \,  d\mu (y)  \biggr] \circ T_{1, i_{l+1} }(x)
	  \, d \nu  (x) \\
	&=   \mu (  C_{i_*}(\tilde{\ell};\tilde{W})  ) \int_X \biggl[
	\tilde{H}( c( \tilde{ \ell }; \tilde{W}  ), \cdot )
	-  \int_X   H( c( \tilde{ \ell }; \tilde{W}  ) , y ) \,  d\mu (y)  \biggr] \circ T_{1, i_{l+1} }(x)
	\, ( d \nu  (x) - d \mu (x) ).
\end{align*}
From the last identity above we see that 
\begin{align*}
|\bar \cI| &\le 2 \mu ( C_{i_*}(\tilde{\ell};\tilde{W})   ) \Vert F \Vert_\infty |  (T_{1, i_{l+1} })_* \nu -  (T_{1, i_{l+1} })_* \mu   |.
\end{align*}
Recalling \eqref{eq:ext_ml}, it follows that 
\begin{align*}
|\bar \cI| &\le 2 \mu ( C_{i_*}(\tilde{\ell};\tilde{W})   ) \Vert F \Vert_\infty |  (\bar{T}_{1, i_{l+1} })_* \bar{\nu} -  (\bar{T}_{1, i_{l+1} })_* \bar{\mu}   | \\
&= 2 \mu ( C_{i_*}(\tilde{\ell};\tilde{W})   ) \Vert F \Vert_\infty |  (\bar{T}_{ \ell_{p + 1} + 1 , i_{l+1} })_* \bar{\nu}_* -  (\bar{T}_{ \ell_{p + 1} + 1 , i_{l+1} })_* \bar{\mu}_*   |,
\end{align*}
where $\bar{\nu}_* = (\bar{T}_{1, \ell_{p+1}})_* \bar{\nu}$ and 
$\bar{\mu}_* = (\bar{T}_{1, \ell_{p+1}})_* \bar{\mu}$. 
Recall also that
$\ell_{p+1} \le i_\# = i_l + \lfloor 2g / 3 \rfloor$.

In order to apply Theorem~\ref{thm:decdec}, we need to show that 
$\bar{\mu}_*$ and $\bar{\nu}_*$ are regular measures. Since $\bar{\mu}$ is regular, 
the regularity of $\bar{\mu}_*$ follows from Proposition \ref{prop:regular}-(c).
Moreover, by Proposition~\ref{prop:onedec_prelim}, $\bar{\mu}_*$ has tail bound 
$r(j) = \frac12 h_{ \ell_{p+1} }^{ 1 }(j) \le C \exp( - C' j )$.

As for $\bar{\nu}_*$, we observe that it is a measure supported on $\bar{Y}$ with 
density given by
\begin{align*}
	\psi(y, 0) = \frac{1}{\mu(   C_{i_*}(\tilde{\ell};\tilde{W})   )} \frac{  \rho( \tilde{y}   )  }{ |
		  \det D  (  T_{1,  \ell_{p + 1} }  |_{  C_{i_*}(\tilde{\ell};\tilde{W})  } ) \tilde{y}   | },
\end{align*}
where $\tilde{y}$ is the unique preimage of $y$ on $C_{i_*}(\tilde{\ell};\tilde{W})$ under $T_{1, \ell_{p + 1}}$.
Since $T_{k, k + R_k(W) - 1} |_W$ is affine for $W \in \cP_k$, it follows by the chain rule that 
$y \mapsto D  (  T_{1,  \ell_{p + 1} }  |_{ C_{i_*} } ) \tilde{y}  $ is constant on $Y$. Therefore, for $y_1, y_2 \in Y$, with respective preimages $\tilde y_1, \tilde y_2$, 
\begin{align*}
| \log \psi (y_1, 0) - \log \psi(y_2, 0) | &\le |  \log \rho( \tilde{y}_1  ) - \log  \rho(  \tilde{y}_2  ) |
\le |\rho|_{\LL} d(   \tilde{y}_1  , \tilde{y}_2  ) 
\le \underline{\rho}^{-1} L d(y_1, y_2), 
\end{align*}
where Proposition~\ref{induce-prop}-(c) was used in the last inequality. 
Thus $| \bar{\nu}_* |_{\LL} \le \underline{\rho}^{-1} L$.
Since $\underline{\rho}^{-1} L < K_2$, 
it follows by Proposition \ref{prop:regular}-(b) that $\bar{\nu}_*$ is regular 
with tail bound $r(j) \le C \bar{h}^{ \ell_{p + 1} }(j) \le C \exp(-jC')$.

Now, Theorem~\ref{thm:decdec} implies 
\begin{align*}
    |  (\bar{T}_{ \ell_{p + 1} + 1 , i_{l+1} })_* \bar{\nu}_* -  (\bar{T}_{ \ell_{p + 1} + 1 , i_{l+1} })_* \bar{\mu}_*   | \le C \exp( - C' ( i_{l+1} -  \ell_{p + 1} ) )
    \le C \exp( - C'_1 g ).
\end{align*}
Summing over all the cylinders, we obtain 
\begin{align}\label{eq:estim_I2}
	\begin{split}
	| \cI_2|  
	&\le  C  \Vert F \Vert_{\theta}  \exp(  - C'  g ) +  C  \Vert F \Vert_{\theta}  (i_l + 1) \lambda^{-\theta t} \\
	&\le C  \Vert F \Vert_{\theta}   (i_l + 1)  \exp(  - C'  \sqrt{g}  ).
	\end{split}
\end{align}

Combining \eqref{eq:estim_I1}, \eqref{eq:estim_I3} and \eqref{eq:estim_I2}, we conclude that 
\begin{align*}
	|\cI| \le |\cI_1| + |\cI_2| + |\cI_3| \le C \Vert F \Vert_{\theta}   (i_l + 1)   \exp( - C'  \sqrt{ i_{l+1} - i_l } ).
\end{align*}
The proof of Theorem~\ref{thm:main-fcb} is complete. \qed

\section{Proof of Theorem~\ref{thm:main-clt}}\label{sec:clt}

We begin by introducing some notation. Throughout the proof, we write
$W$ for $W_N$, $\Sigma$ for $\Sigma_N$, and $S$ for $S_N$, omitting the subscripts.
For an $\bR^d$-valued random vector $\xi$ on
$(X,\mu)$, we write
$
\mu(\xi) = \int_X \xi \, d\mu
$, $\overline{\xi} = \xi - \mu(\xi)$,
and $[\xi]_r$ for the components of $\xi$ ($r = 1,\ldots,d$).
The spectral norm of a matrix $A \in \bR^{d \times d}$ is denoted by 
$\Vert A \Vert_\spe$. 
For any $0 \le n < N$ and $m \in \bZ$, we define
\begin{align*}
	Y_n &= B^{-1}  \xi_n, \qquad B = \Sigma^{1/2}, \\
	W_{n,m} &= \sum_{\substack{0 \le i < N \\ |i-n| > m}} Y_i, \qquad
    \xi_{n,m} = \sum_{\substack{0 \le i < N \\ |i-n| = m}}  \xi_i,
	\qquad
	Y_{n,m} = B^{-1}  \xi_{n,m}.
\end{align*}
Note that $W_{n,N-1} = 0$ and $W_{n,-1} = W$. 

Without loss of generality we assume that 
\begin{align}\label{eq:wlog_bound}
\lambda_{\min}^{-1/2}(N) = 
\Vert B^{-1} \Vert_\spe \le \frac{1}{100 L}.
\end{align}
Otherwise, we have the trivial upper bound 
\begin{align}\label{eq:trivial}
\begin{split}
d_{\cW}(  W, Z  ) &\le  \mu(  | W |  ) + E( | Z | )  \le  N 2L \Vert B^{-1} \Vert_\spe + 
E( | Z | )
\\
&\le   C_0 L^3  N \Vert B^{-1} \Vert_\spe^3 + 
 C_0 \sqrt{d} L \Vert B^{-1} \Vert_\spe,
\end{split}
\end{align}
for some absolute constant $C_0 > 0$, which implies \eqref{eq:w1_bound}.

The proof of Theorem~\ref{thm:main-clt} follows the approach of \cite{LS20},
which is based on Stein's method. Given $h \in \mathcal{W}$, we consider the
Stein equation associated with the multivariate standard normal distribution
$N_d(0,\bI_{d})$,
\begin{align}\label{eq:se}
	\Delta A(w) - w^{ \mathrm{T} } \nabla A(w)
	= h(w) - E(h(Z)) .
\end{align}
Our goal is to control the right-hand side of \eqref{eq:se} after substituting
$w = W$ and taking expectations with respect to $\mu$.

By \cite[Proposition~2.2]{GMS18}, the solution to \eqref{eq:se}, whose explicit
formula is given by
\begin{align*}
	A_h(w)
	= -\int_0^{\infty}
	\left\{
	E \left[ h \left(e^{-s}w + \sqrt{1 - e^{-2s}}\, Z \right) \right]
	- E(h(Z))
	\right\}
	\, ds ,
\end{align*}
satisfies $A_h \in C^2(\bR^d,\bR)$, and the 
second derivative satisfies the
($1 + \log$) Lipschitz bound 
\begin{align}\label{eq:log_lip}
\| D^2 A_h(w) - D^2 A_h(w') \|_{\rm sp} \le 	| w - w' |  (  C_\#(d) +   | \log | w - w' | | ),
\end{align}
where
\[
C_\#(d)
= 2^{\frac{3}{2}}
\frac{1 + 2d}{d}
\frac{\Gamma\!\left( \frac{1 + d}{2} \right)}{\Gamma\!\left( \frac{d}{2} \right)} \le 6 \sqrt{d},
\]
and $\Gamma$ denotes the gamma function. We define
\[
\mathcal{A} = \{ A_h \colon h \in \mathcal{W} \}.
\]

Next, for $A \in C^2(\bR^d, \bR)$, $u \in [0,1]$, $0 \le n < N$ and $m \in \bZ$, we set 
\[
\delta^{n,m}(u) = D^2A(W_{n,m} + u\, Y_{n,m}) - D^2A(W_{n,m})
\] 
and 
\[
\delta^{n,m} = \delta^{n,m}(1) = D^2 A(W_{n,m-1}) - D^2A(W_{n,m}).
\]
By \cite[Proposition~5.3]{LS20}, 
\[
\mu[ \Delta A(W) - W^\mathrm{T} \nabla A(W)] = \sum_{i=1}^7 E_i, 
\]
where
\begin{align*}
	E_1 & = - \sum_{n=0}^{N-1}\sum_{m=1}^{N-1} \int_0^1  \, \mu[  Y_{n}^\mathrm{T} \delta^{n,m}(u) Y_{n,m}   ] \, du, \quad 
	E_2 = - \sum_{n=0}^{N-1} \int_0^1   \mu[ Y_{n}^\mathrm{T} \delta^{n,0}(u) Y_n ] \, du,
	\\
	E_3 & = - \sum_{n=0}^{N-1}\sum_{m=1}^{N-1} \sum_{k=m+1}^{2m}    \mu [  Y_{n}^\mathrm{T} \, \overline{\delta^{n,k}} \, Y_{n,m}  ],
	\quad 
	E_4  = - \sum_{n=0}^{N-1}\sum_{m=1}^{N-1} \sum_{k=2m+1}^{N-1}     \mu [  Y_{n}^\mathrm{T} \, \overline{\delta^{n,k}} \, Y_{n,m}  ],
	\\
	E_5 & = - \sum_{n=0}^{N-1} \sum_{k=1}^{N-1}  \mu [  Y_{n}^\mathrm{T} \, \overline{\delta^{n,k}} \, Y_n  ],
	\quad 
	E_6  = \sum_{n=0}^{N-1}\sum_{m=1}^{N-1}  \mu \left[  Y_{n}^\mathrm{T}   \sum_{k=0}^m \mu( \delta^{n,k} )  Y_{n,m}  \right],
	\\
	E_7 & =  \sum_{n=0}^{N-1}  \mu [  Y_{n}^\mathrm{T}   \mu( \delta^{n,0} )  Y_n  ].
\end{align*}
Therefore,
\begin{align}\label{eq:prelim}
	d_{ \cW }(W, Z) \le \sup_{ A \in \cA  } \sum_{i=1}^7 |E_i|.
\end{align}
It remains to control each term $E_i$ for $A \in \cA$. To this end, we note that the bound in Theorem~\ref{thm:main-fcb}  extends readily to the case 
of \say{several gaps} by induction as follows (see \cite[Proposition~2.9]{LNN25}): Suppose that $0 \le i_1 < \cdots < i_k \le n$, and $1 \le \ell_1 < \cdots < \ell_p < k$ are integers, $k \ge 2$, 
and that $F \colon X^k \to \bR$ is  separately $\eta$-H\"older continuous for some $\eta \in (0,1]$. Define
\begin{align*}
	&H(x_1, \ldots, x_{p+1}) \\
	&= F( T_{1, i_1}(x_1), \ldots , T_{1, i_{ \ell_1 } }(x_1),   T_{1, i_{ \ell_1 + 1 } }(x_2), \ldots, T_{1, i_{ \ell_2 } }(x_2 ), 
	\ldots, T_{ 1, i_{ \ell_p + 1 } } (x_{p+1}), \ldots T_{1, i_k}(x_{p+1}) ).
\end{align*}
Then, 
\begin{align}\label{eq:fcb_many_gaps}
	\begin{split}
		&\biggl| \int_X H(x,\ldots, x) \, d \mu(x) - \int\cdots\int_{X^{p+1}} H(x_1, \ldots, x_{p+1}) \, d \mu(x_1) \cdots   \, d \mu(x_{p+1}) \biggr| \\
		&\le  C \Vert F \Vert_{ \eta } \sum_{j=1}^p ( i_{ \ell_j }  + 1 )  \exp( -C' \sqrt{ i_{\ell_j + 1} - i_{ \ell_j }  }   ).
	\end{split}
\end{align}

\subsection{Bounds on $E_i$} Let $C,C'$ denote various constants depending only on 
$\theta, \underline{\rho}, | \rho |_{ \Lip }$ and  $M, \overline{a}, \underline{a}$. For convenience, we introduce the following conventions for sequences of points in $X$.
For any subset $J \subset \mathbb{Z}_+$, 
let $X^J$ denote the set of functions from $J$ to $X$, and write $x_J$
for an element $(x_j)_{j \in J}$ in $X^J$.
Using an analogous convention, we write 
$\tilde T_{J} = (  T_{ 1,j }  )_{  j \in J }$.
For two subsets $J_1, J_2 \subset \mathbb{Z}_+$ that are ordered in the sense that 
every element of $J_1$ is less than every element of $J_2$, 
we write $(x_{J_1}, x_{J_2})$ for the concatenation of the corresponding 
vectors. If $J_1 = \emptyset$, we identify $(x_{J_1}, x_{J_2})$ with $x_{J_2}$, and 
more generally, whenever a concatenation involves an empty index set,
the corresponding vector is omitted. Finally, for any integers $p,q$ we define 
subsets of $\bZ_+$ as follows:
\begin{align*}
I(p,q) &= \{ 0 \le j < N \colon p \le j \le q \}, 
\quad I_{-}(p) = \{ 0 \le j < N \colon j \le p \}, \\
I_+(q) &= \{ 0 \le j < N \colon j \ge q \}.
\end{align*}

Given $A \in \cA$, $u \in [0,1]$ and integers  $0 \le m \le N -1$, $0 \le n \le N -1$,
let 
\[\begin{split}
	\Psi_u( x_{ I_{-}(n-m)  }, x_{ I_{+}(n+m) }; n,m ) =& D^2 A \biggl(  B^{-1} \sum_{  \substack{ 0 \le i < N \\ |i - n| > m  }   } \bar{g}_i(x_i)  
	+ u  B^{-1} \sum_{  \substack{ 0 \le i < N \\ |i - n| = m  }   } \bar{g}_i(x_i)  
	\biggr) \\
	&- D^2 A \biggl(    B^{-1} \sum_{  \substack{ 0 \le i < N \\ |i - n| > m  }   } \bar{g}_i(x_i)  \biggr).
\end{split}\]
If $m=0$, then $n$ belongs to both $I_-(n-m)$ and $I_+(n+m)$.
In this case, we interpret
$(x_{I_-(n-m)},x_{I_+(n+m)})$
as $(x_0,\ldots,x_{N-1})$, with the coordinate $x_n$ appearing only once.
Note that 
\[
\Psi_u(  \tilde{T}_{ I_-(n-m)  },   \tilde{T}_{ I_+(n+m)  } ; n,m ) = \delta^{n,m}(u).\]
Inequality \eqref{eq:log_lip} implies 
\begin{align*}
	&\Vert 	\Psi_u( x_{ I_{-}(n-m)  }, x_{ I_{+}(n+m) } ; n,m ) \Vert_\spe  \\
	&\le C
	\biggl| u B^{-1} \sum_{  \substack{ 0 \le i < N \\ |i - n| = m  }  } \bar{g}_i(x_i) \biggr|
	\biggl[  C_\#(d) -   \log \biggl| u B^{-1} \sum_{  \substack{ 0 \le i < N \\ |i - n| = m  }  } \bar{g}_i(x_i) \biggr|  \biggr].
\end{align*}
By \eqref{eq:wlog_bound}, we have
\begin{align*}
    \biggl| u B^{-1} \sum_{  \substack{ 0 \le i < N \\ |i - n| = m  }  } \bar{g}_i(x_i) \biggr|
    < e^{-1}.
\end{align*}
Using the property that $t \mapsto t \log (1/t)$ is increasing 
on $(0, e^{-1}]$, we obtain 
\begin{align}\label{eq:psi_sup}
	\Vert 	\Psi_u( x_{ I_{-}(n-m)  }, x_{ I_{+}(n+m) } ; n,m ) \Vert_\spe  
	\le C L \Vert B^{-1} \Vert_\spe  ( C_\#(d)  - \log \Vert B^{-1} \Vert_\spe   ).
\end{align}

For brevity, we write $x_{ J } ( y / j )$ for the sequence in $X^J$ obtained from
$x_{J}$ by replacing $(  x_J )_j$ 
with $y \in X$. Using \eqref{eq:log_lip}, the H\"older continuity of $g_j$, 
and the property that $t \mapsto t \log (1/t)$ is increasing 
on $(0, e^{-1}]$,
we obtain for $y, y' \in X$
and $j \in I_{-}(n-m)$ the upper bound
\begin{align}\label{eq:lip_psi-1}
\begin{split}
  &\Vert 	\Psi_u( x_{ I_{-}(n-m)  } ( y / j )
	, x_{I_{+}(n+m)} ; n,m )  -  	\Psi_u( x_{ I_-(n-m)  } ( y' / j ) , x_{I_+(n+m)} ; n,m )  \Vert_\spe  \\
  &\le C | B^{-1} (  g_j( y ) - g_j(y')  ) |  [ C_\#(d) - \log | B^{-1} (  g_j( y ) - g_j(y')  ) | ]   \\
  &\le C \Vert B^{-1} \Vert_\spe L ( C_\#(d) - \log \Vert B^{-1} \Vert_\spe ) |y - y'|^\theta  
  -  C \Vert B^{-1} \Vert_\spe L |y - y'|^\theta  \log( L |y-y'|^\theta ) \\
  &\le C \Vert B^{-1} \Vert_\spe L ( C_\#(d) - \log \Vert B^{-1} \Vert_\spe ) |y - y'|^{\theta/2}.
\end{split}
\end{align} 
Similarly, for $y, y' \in X$
and $j \in I_{+}(n+m)$, 
\begin{align}\label{eq:lip_psi-2}
\begin{split}
  &\Vert 	\Psi_u( x_{ I_{-}(n-m)  } 
	, x_{I_{+}(n+m)} ( y / j ) ; n,m )  -  	\Psi_u( x_{ I_-(n-m)  }, x_{I_+(n+m)}( y' / j ) ; n,m )  \Vert_\spe  \\
  &\le C \Vert B^{-1} \Vert_\spe L ( C_\#(d) - \log \Vert B^{-1} \Vert_\spe ) |y - y'|^{\theta/2}.
\end{split}
\end{align}

The upper bounds \eqref{eq:psi_sup}, \eqref{eq:lip_psi-1}, and \eqref{eq:lip_psi-2} enable 
us to apply \eqref{eq:fcb_many_gaps} to bound $E_1$. To this end, 
for $m < k$ and $u \in [0,1]$ we define 
\begin{align*}
	&F_u(  x_{ I_{-}(n-k)}, x_{n-m}, x_n, x_{n + m}, x_{ I_{+}(n+k)}     ; n,m, k ) \\
	&=  (  B^{-1} \bar{g}_n (x_n) )^{ \mathrm{T} }   \Psi_u (  x_{ I_{-}(n-k) }  , 
	x_{ I_{+}(n+k) }
    ; n,k
	) \sum_{ \substack{ 0 \le i < N \\  |i - n| = m }  } B^{-1} \bar{g}_i(x_i).
\end{align*}
If $m=k$, then we define
$F_u(x_{I_-(n-m)},x_n,x_{I_+(n+m)};n,m, m)$
by the same formula, omitting the arguments $x_{n-m}$ and $x_{n+m}$,
since they are already included in $I_-(n-m)$ and $I_+(n+m)$. 

For brevity, let us denote
\[
\fD
=
L^3 \Vert B^{-1} \Vert_\spe^3
(
C_\#(d)-\log \Vert B^{-1} \Vert_\spe
).
\]
Combining \eqref{eq:psi_sup} and \eqref{eq:lip_psi-1} and \eqref{eq:lip_psi-2}, it is straightforward to verify that, for any $m \le k$ and $u \in [0,1]$, 
\begin{align*}
	\Vert F_u(\cdot; n, m, k) \Vert_{ \theta / 2 } \le C \fD.
\end{align*}
Moreover,
\begin{align*}
    &F_u(  \tilde{T}_{ I_{-}(n-k)}, T_{1, n-m}, T_{1,n}, T_{1, n + m}, \tilde{T}_{ I_{+}(n+k)}  ; n,m, k ) = Y_{n}^{ \mathrm{T} } \delta^{n,k}(u) Y_{n,m}, \quad m < k, \\
    &F_u(  \tilde{T}_{ I_{-}(n-m)}, T_{1,n}, \tilde{T}_{ I_{+}(n+m)}  ; n,m,m ) = Y_{n}^{ \mathrm{T} } \delta^{n,m}(u) Y_{n,m}.
\end{align*}
Since $\mu(  \bar{g}_n( T_{1,n} )  ) = 0$,  \eqref{eq:fcb_many_gaps} with $\eta = \theta/2$ implies that,
for any $m \le k$ and $u \in [0,1]$, 
\begin{align}\label{eq:fcb_appli1}
	\begin{split}
	&|	\mu[  Y_{n}^{ \mathrm{T} } \delta^{n,k}(u) Y_{n,m}   ] | 
	\le C \fD n e^{ - C' \sqrt{m} },
	\end{split}
\end{align}
and 
\begin{align}\label{eq:fcb_appli2}
	|	\mu[  Y_{n}^{\mathrm{T}}  \overline{\delta^{n,k}(u)} Y_{n,m}   ] |
	\le C \fD
	\min\{  n e^{ - C' \sqrt{m} }, (n+ m)  e^{ - C' \sqrt{ k - m} }
	 \}.
\end{align}

Consider $E_1$. For $\rho_N \in \bN$,  
by \eqref{eq:fcb_appli1} we have 
\begin{align*}
	|E_1| \le& \sum_{n=0}^{N-1}\sum_{m=1}^{N-1} \int_0^1  | \mu[ Y_{n}^{ \mathrm{T} } \delta^{n,m}(u) Y_{n,m}   ]  | \, du \\
	\le& \sum_{n=0}^{N-1}\sum_{m=1}^{\rho_N - 1} \int_0^1  | \mu[  Y_{n}^{ \mathrm{T} } \delta^{n,m}(u) Y_{n,m}   ]  | \, du 
	+ \sum_{n=0}^{N-1}\sum_{m= \rho_N }^{N-1} \int_0^1  | \mu[  Y_{n}^{ \mathrm{T} } \delta^{n,m}(u) Y_{n,m}   ]  | \, du \\
	\le& C \fD N \rho_N
	+   C \fD  N^2  \sum_{m= \rho_N }^{\infty} e^{ - C' \sqrt{m} } 
	\le C \fD  ( N \rho_N  + N^2 e^{ - ( C' / 2) \sqrt{ \rho_N }    } ).
\end{align*}
Choosing $\rho_N =  \lceil ( C_1  \log N )^2 \rceil$ for $C_1$ sufficiently large (depending on $C'$) yields 
\begin{align*}
	|E_1| &\le C \fD N (1 + \log^2 N).
\end{align*}

Further, using \eqref{eq:psi_sup},
\begin{align*}
	|E_2| + |E_7| &= \biggl|  \sum_{n=0}^{N-1} \int_0^1   \mu[ Y_{n}^{ \mathrm{T} }\delta^{n,0}(u) Y_n ] \, du \biggr| 
	+ \biggl|  \sum_{n=0}^{N-1}  \mu [  Y_{n}^{ \mathrm{T} }  \mu( \delta^{n,0} )  Y_n  ] \biggr|
	  \le C \fD N.
\end{align*}

As in the case of $E_1$, for $\rho_N \in \bN$, using \eqref{eq:fcb_appli2} we obtain
\begin{align*}
	&|E_3| + |E_4| \\
    =& \biggl| \sum_{n=0}^{N-1}\sum_{m=1}^{N-1} \sum_{k=m+1}^{2m}    \mu [  Y_{n}^{ \mathrm{T} } \, \overline{\delta^{n,k}} \, Y_{n,m}  ] 
	\biggr| + \biggl|  \sum_{n=0}^{N-1}\sum_{m=1}^{N-1} \sum_{k=2m+1}^{N-1}     \mu [  Y_{n}^{ \mathrm{T} } \, \overline{\delta^{n,k}} \, Y_{n,m}  ] 
	\biggr| \\
	\le& \biggl| \sum_{n=0}^{N-1}\sum_{m=\rho_N}^{N-1} \sum_{k=m+1}^{2m}    \mu [  Y_{n}^{ \mathrm{T} } \, \overline{\delta^{n,k}} \, Y_{n,m}  ] 
	\biggr|
	+ \biggl| \sum_{n=0}^{N-1}\sum_{m= 1 }^{ \rho_N -1} \sum_{k=m+1}^{2m}    \mu [  Y_{n}^{ \mathrm{T} } \, \overline{\delta^{n,k}} \, Y_{n,m}  ] 
	\biggr|
	\\
	&+ \biggl|  \sum_{n=0}^{N-1}\sum_{m= \rho_N }^{N-1} \sum_{k=2m+1}^{N-1}     \mu [  Y_{n}^{ \mathrm{T} } \, \overline{\delta^{n,k}} \, Y_{n,m}  ] 
	\biggr|  \\
	&+ \biggl|  \sum_{n=0}^{N-1}\sum_{m= 1  }^{\rho_N  - 1} \biggl\{  
	\sum_{k=2m+1}^{  2 \rho_N  }     \mu [  Y_{n}^{ \mathrm{T} } \, \overline{\delta^{n,k}} \, Y_{n,m}  ]  + 
	\sum_{k = 2 \rho_N + 1 }^{N-1}     \mu [  Y_{n}^{ \mathrm{T} } \, \overline{\delta^{n,k}} \, Y_{n,m}  ]
	\biggr\}
	\biggr| \\
	\le&  C \fD 
    \biggl\{  
	 N^2  \sum_{ m = \rho_N }^{N-1} m e^{ - C' \sqrt{m} } +  N \rho_N^2  
     +
	 N^2 \sum_{ m  = 1}^{  \rho_N - 1 }  \sum_{ k = 2 \rho_N + 1 }^{N-1}  e^{ - C' \sqrt{ k - m} } \biggr\}.
\end{align*}
Choosing $\rho_N =  \lceil ( C_2  \log N )^2 \rceil$ for $C_2$ sufficiently large, it follows that 
\begin{align*}
	|E_3| + |E_4| \le C \fD N (  1 + \log^4N ).
\end{align*} 
A similar application of \eqref{eq:fcb_appli2} gives
\begin{align*}
	|E_5| + |E_6| \le C \fD N (  1 + \log^2N ).
\end{align*} 
Recalling that $C_\#(d) \le C \sqrt{d}$, we have 
$$
\fD \le C \sqrt{d} L^3 
\max\{1 , \log \lambda_{ \min } (N) \} \lambda_{ \min }^{-3/2}.
$$
We conclude that 
\begin{align*}
	\sup_{A \in \cA} \sum_{i=1}^7 |E_i| \le C L^3  \sqrt{d} \max\{1 , \log \lambda_{ \min } (N) \}   \lambda_{ \min }^{-3/2} (N) N (  1 + \log^4N ).
\end{align*}
Combined with \eqref{eq:prelim} and \eqref{eq:trivial}, this gives the desired estimate
\eqref{eq:w1_bound}. \qed


\section{Proof of Theorem \ref{thm:nearby}}\label{sec:nearby}

    Let $(a_k)_{k \ge 1}$ be a sequence of parameters satisfying \eqref{eq:param_range}.
    For each $a\in(0,\frac{1}{M})$,
    let $\cL_a \colon L^1(m) \to L^1(m)$ denote the transfer operator associated to $(f_a, m)$,
    satisfying for every $\psi \in L^1(m)$ and $\varphi \in L^\infty(m)$ the duality relation
    \[
    \int_X \varphi \circ f_a \psi \, d m = \int_X \varphi  \cL_a(\psi) \, d m.
    \]
    A direct computation shows that
	$\cL_a$ can be represented explicitly as 
	\begin{align}\label{eq:to}
    \begin{split}
		\cL_a u( x_u, x_c ) &= \sum_{i=1}^M  \mathbf{1}_{  [  ( i - 1 ) / M, i / M  ) } (x_c) \cdot aM
		u(  a x_u + (i -1) a, M x_c - (i - 1)  ) \\
		&\quad+ \sum_{i=1}^M \frac{1 - Ma}{M} u \biggl( (1 - Ma) x_u + Ma , \frac{1}{M} (x_c + i - 1) \biggr),
    \end{split}
	\end{align}
where $\mathbf{1}_{  [  ( i - 1 ) / M, i / M  )  }$ denotes the indicator function for the interval $[  ( i - 1 ) / M, i / M  )$.
	
	As in the case of maps, we denote time-dependent compositions of transfer operators by
	\begin{align*}
		\cL_{j,k}  = \begin{cases}
			\cL_{a_k} \circ \cdots \circ \cL_{a_j}, &\text{if $1 \le j \le k$,} \\
			\text{id}_{L^1(m)}, &\text{otherwise.}
		\end{cases}
	\end{align*}
	For $L \ge 1$, Theorem~\ref{thm:main-ml} implies that there exist constants $C_L, C' > 0$ such that for any function $\psi\colon X \to \bR$ with $\int \psi \, dm = 0$
    and $\Vert \psi \Vert_{\Lip} \le L$, and for any $k \ge 1$,
	\begin{align*}
		\Vert  \cL_{k, k + n -1}( \psi )  \Vert_{L^1(m)} \le  C_L e^{ - C' n }.
	\end{align*}
	Let $p \in [1, \infty)$. Using $\Vert \cL_a u \Vert_{ L^\infty(m) } \le \Vert u \Vert_{ L^\infty(m) }$, 
    the previous bound extends to
	\begin{align}\label{eq:ml_lp}
		\Vert  \cL_{k, k + n -1}( \psi )  \Vert_{L^p(m)} \le  C_{p, L} e^{ - C' n / p }.
	\end{align}
    The constant $C'$ depends only on $\theta, M, \bar{a}, \underline{a}$ and $C_{p,L}$ in addition on 
    $L$ and $p$.

	\subsection{Perturbation of transfer operator} 
	
	We verify that $a \mapsto \cL_a u$ is Lipschitz continuous in $L^\infty(m)$ for suitably regular functions $u$. To this end we define 
	\[
	\Vert u \Vert_{*} = \Vert u \Vert_{ L^\infty(m) } + \Vert  \partial_{ x_u } u  \Vert_{L^\infty(m)}
	\]
	and 
	\[
	\cS = \{  u \in L^\infty(m) \colon    \partial_{ x_u } u \in L^\infty(m)  \},
	\]
    where $\partial_{ x_u } u$ denotes the weak partial derivative of $u$ with respect to the $x_u$-coordinate.

	\begin{prop} Let $I$ be a closed interval in $(0, \frac{1}{M})$.
		For all $a,b \in I$, and all $u \in \cS$, 
		\begin{align}\label{eq:perturb}
			\Vert  (  \cL_a - \cL_b  ) u  \Vert_{L^\infty(m)} \le  6 M \Vert u \Vert_* |a-b|.
		\end{align}
		Moreover, for any $a \in I$ and any $u \in \cS$, 
		\begin{align}\label{eq:contract_star}
			\Vert \cL_a u \Vert_* \le \Vert  u \Vert_*.
		\end{align}
	\end{prop}
	
	\begin{proof} Let $u \in \cS$, 
    $a,b \in I$, $x_u \in (0,1)$ and $(i-1)/M \leq x_c < i/M$ with $i \in \{1,\ldots, M\}$. 
		From \eqref{eq:to} we obtain
		\begin{align*}
			&| (  \cL_a - \cL_b  ) u(x_u, x_c)  | \\
			&\le  \biggl|  aM
			u(  a x_u + (i -1) a, M x_c - (i - 1)  )
			\\
			\qquad &- bM
			u(  bx_u + (i -1) b, M x_c - (i - 1)  ) \biggr| \\
			&+ \sum_{i=1}^M \biggl| \frac{1 - Ma}{M} u \biggl( (1 - Ma) x_u + Ma , \frac{1}{M} (x_c + i - 1) \biggr) \\
			&\quad - 
			\frac{1 - Mb}{M} u \biggl( (1 - Mb) x_u + Mb , \frac{1}{M} (x_c + i - 1) \biggr) \biggr| =: I + I\!I.
		\end{align*}
        Since $u$ has a representative satisfying
        \[
        | u(x_u, x_c) - u(x_u', x_c) | \le \Vert \partial_{x_u} u \Vert_{L^\infty(m)} |x_u - x_u'|\]
        for $m$-almost every $x_c \in [0,1]$, and all $x_u, x_u' \in [0,1]$, we obtain
		\[
		I \le 3 M \Vert u \Vert_* |a-b| \quad \text{and} \quad I\!I \le 3 M \Vert u \Vert_* |a-b|,\]
		which implies \eqref{eq:perturb}.
		
		Further, from \eqref{eq:to} it is straightforward to verify that
		\begin{align*}
			\Vert \cL_a u \Vert_{L^\infty(m)} \le aM \Vert u \Vert_{L^\infty} + (1 - Ma) \Vert u \Vert_{L^\infty} \le 
			\Vert u \Vert_{L^\infty}
		\end{align*}
		and
		\begin{align*}
			\Vert \partial_{x_u} \cL_a u \Vert_{L^\infty(m)} \le ( a^2 M  + (1 - Ma)^2 ) \Vert \partial_{x_u}  u \Vert_{L^\infty} \le   \Vert \partial_{x_u}  u \Vert_{L^\infty}.
		\end{align*}
		which gives \eqref{eq:contract_star}.
	\end{proof}
	
	\subsection{Linear growth of $\lambda_{\min}(N)$} 
    
    There exists $\ve_{a,M} \in (0, \min\{ a, \frac{1}{2M} - a \} )$ such that 
    \eqref{eq:param_range} holds for any 
    sequence $(a_k)_{k \ge 1}$
    of parameters in $(0,\frac{1}{2M})$ with $a_k \in (a - \ve , a + \ve)$, whenever $0 < \ve \le \ve_{a, M}$.
    In the remainder of this section we assume that 
    $0 < \ve \le \ve_{a, M}$ and only consider such sequences.
    
    To establish 
    Theorem~\ref{thm:nearby},
	it suffices to show \eqref{eq:lambda_growth}, for then \eqref{eq:rate_nearby} follows 
	from \eqref{eq:w1_bound}. Below we prove \eqref{eq:lambda_growth} using the argument from \cite[Theorem~4.1]{NTV18}, which involves approximating Birkhoff sums by reverse martingales.
	
	Consider the setting in Theorem~\ref{thm:nearby}. Let $v \in \bR^d$ be an arbitrary unit vector. Note that 
	\begin{align*}
		v^{ \mathrm{T} } \Sigma_N v = \int S_{N,v}^2 \, dm,
	\end{align*}
	where 
	\[
	S_{N,v} = \sum_{k=0}^{N-1} \bar{g}_v \circ T_{1,k}, \quad \bar{g}_v = g_v - \int g_v \, dm, \quad 
	g_v = \langle v,g\rangle.\]
	By our assumption, $g_v$ cannot be written as $\varphi_v \circ f_a - \varphi_v + c_v$
	for any $c_v \in \bR$ and $\varphi_v \in L^2(m)$. Moreover,
	\[
	\lambda_{\min}(N) = \inf_{ \substack{ v\in\mathbb R^d \\ | v | = 1 } } v^{ \mathrm{T} } \Sigma_N v = \inf_{v\in\mathbb R^d \\ | v | = 1  } \int S_{N,v}^2 \, dm,\]
    where $|v|$ denotes the Euclidean norm of $v\in\mathbb R^d$.
	
	Let $G_{0,v} = 0$ and for $n \ge 1$ define 
	\[
	G_{n,v} = \sum_{j = 1}^n \cL_{j,n}( \bar{g}_v ).
	\] 
	Further, for $n \ge 0$ let 
	\[
	H_{n,v} = \bar{g}_v + G_{n,v} - G_{n+1,v} \circ T_{n+1}.
	\]
	Then
	\begin{align}\label{eq:mgle_decomp}
		S_{n,v} = \sum_{k=0}^{n-1} \bar{g}_v \circ T_{1,k} = \sum_{k=0}^{n-1} H_{k,v} \circ T_{1,k} + G_{n,v} \circ T_{1,n}.
	\end{align}
	Let $\cF_n = f_{1,n}^{-1}( \cF_0 )$ where $\cF_0$ denotes the sigma-algebra of Borel sets on $X$.
	As in \cite[Section~2]{NTV18}, we have 
	\begin{align}\label{eq:reverse_mgle}
		\bE[   H_{n,v} \circ T_{1,n}  \mid  \cF_{n + 1}  ] = 
        (\cL_{ a_{n+1} } H_{n,v} ) \circ T_{1, n + 1} = 0,
	\end{align}
	where $\bE[ \psi \mid \cB ]$ denotes the conditional expectation 
    of $\psi$
    with respect to a sub-sigma-algebra 
	$\cB$ and the measure $m$. In particular, $(H_{n,v} \circ T_{1,n})$ is a reverse 
    martingale difference sequence with respect to the decreasing filtration $(\cF_n)$.
    Let 
    \[
    M_{n,v} = \sum_{k=0}^{n-1} H_{k,v} \circ T_{1,k}
    \]
    denote the associated partial sum process. In the sequel, we also write $\bE$ for the expectation 
    with 
	respect to $m$. 
	
	Using \eqref{eq:mgle_decomp} and \eqref{eq:reverse_mgle}, we have 
	\begin{align*}
	\bE[ S_{n,v}^2 ]	 =  \bE[ M_{n,v}^2 ] + \bE[  G_{n,v}^2  ].
	\end{align*}
    Since $\Vert \langle v, g \rangle\Vert_{\Lip} \le \Vert g \Vert_{\Lip}$, 
    \eqref{eq:ml_lp} implies 
	\[
    \bE[  G_{n,v}^2  ] \le C_g
	\]
    for some constant $C_g > 0$ independent of $v$.
	Consequently, by the orthogonality property of martingales,
	\begin{align*}
		\bE[ S_{n,v}^2 ] = \sum_{k=0}^{n-1}  \bE[ H_{k,v}^2 ] + O(1),
	\end{align*}
    where the error term is uniform in $v$.
	
	Next, we write $\tilde \cL$ for the transfer operator with respect to $( \tilde{f}, m)$
	where $\tilde{f} = f_a$. We set $\tilde \cL^n = \tilde \cL \circ \tilde \cL^{n-1}$ 
    and $\tilde{f}^n = \tilde{f} \circ \tilde{f}^{n-1}$
	for $n \ge 1$, where $\tilde \cL^0 = \text{id}_{L^1(m)}$ and $\tilde{f}^0 = \text{id}_{X}$.
    Let
	\[
	\tilde{H}_v = \bar{g}_v + \tilde{G}_v - \tilde{G}_v \circ \tilde{f}
	\]
	where the series
	\[
	\tilde{G}_v = \sum_{k=1}^\infty \tilde{\cL}^k ( \bar{g}_v )
	\]
	converges in $L^p$ for any $p \in [1, \infty)$. We have 
	\[
	\tilde{S}_{n,v} := \sum_{k=0}^{n-1} \bar{g}_v \circ \tilde{f}^k = \sum_{k=0}^{n-1} \tilde{H}_v \circ \tilde{f}^k 
	- \tilde{G}_v + \tilde{G}_v \circ \tilde{f}^n,
	\]
	and 
	\[
	\bE[  \tilde{H}_v \circ \tilde{f}^n  \mid   \tilde{\cF}_{n + 1}   ] = 0
	\]
	where $\tilde{\cF}_n = (  \tilde{f}^n )^{-1} \cF_0 $. Note that $\tilde{H}_v \neq 0$
    in $L^2(m)$ for every unit vector $v\in\mathbb R^d$, because $g$ is not a coboundary for $\tilde{f}$
    in any direction. In particular,
    \begin{align*}
        \kappa := \inf_{ \substack{v\in\mathbb R^d \\ | v | = 1  } }\bE[ \tilde{H}_v^2  ] > 0.
    \end{align*}
    
    Similarly to the case of $S_n$ we have
	\begin{align}\label{eq:variange_mgle}
		\bE[ \tilde S_{n,v}^2 ]	= \sum_{i=1}^n \bE[ \tilde H_v^2 \circ \tilde{f}^i  ] + O(1) = n \bE[ \tilde H_v^2 ] + O(1),
	\end{align}
    where the error is uniform in $v$.
	Therefore, there exists a constant $C > 0$ such that for 
    every unit vector $v$ and $n \ge 1$, 
	\begin{align*}
		| \bE[  S_{n,v}^2 ]	-  \bE[ \tilde S_{n,v}^2 ]	| \le C + \biggl|  n \bE[ \tilde H_v^2 ] - \sum_{k=0}^{n-1} \bE[ H_{k,v}^2 ]    \biggr| \le 
        C + \sum_{k=0}^{n-1} |  \bE[ \tilde H_v^2 ] -  \bE[ H_{k,v}^2 ] |.
	\end{align*}
	
	Let $\ve_1 > 0$. In the remainder of the proof we denote by $C$ various 
    positive constants independent of $n$, $\ve, \ve_1$ and $v$.
    By \eqref{eq:ml_lp}, 
	there exists $N(\ve_1) \in \bN$ 
    such that for every $n \ge N(\ve_1)$ and every unit vector $v\in\mathbb R^d$,  
	\begin{align*}
		G_{n,v} = \sum_{j=0}^{N(\ve_1)-1} \cL_{n - j,n} ( \bar{g}_v ) + E_v(\ve, n, N(\ve_1) )
	\end{align*}
	and 
	\begin{align*}
		\tilde G_v =  \sum_{k=1}^{ N(\ve_1) } \tilde{\cL}^k ( \bar{g}_v ) + \tilde E_v(\ve,  N(\ve_1) )
	\end{align*}
	where $\sup_{\substack{v\in\mathbb R^d \\ | v | = 1 }} \Vert E_v(\ve, n, N(\ve_1) ) \Vert_{L^2} < \ve_1$ and $\sup_{\substack{v\in\mathbb R^d \\ | v | = 1 }} \Vert \tilde E_v (\ve, N(\ve_1) ) \Vert_{L^2} < \ve_1$. On the other hand, 
	\begin{align*}
		\tilde H_v^2 = \bar{g}_v^2 - 2 \bar{g}_v ( \tilde G_v \circ \tilde{f} - \tilde{G}_v  ) 
		+ \tilde{G}_v^2 - 2 \tilde{G}_v \cdot \tilde{G}_v \circ \tilde{f} + \tilde{G}_v^2 \circ \tilde{f},
	\end{align*}
	and 
	\begin{align*}
		H_{k,v}^2 = \bar{g}_v^2 - 2 \bar{g}_v (  G_{k+1,v} \circ T_{k+1} - G_{k,v}  ) 
		+ G_{k,v}^2 - 2  G_{k,v}  \cdot G_{k+1,v} \circ T_{k+1}  + G_{k+1,v}^2 \circ T_{k+1}.
	\end{align*}
	
	Recall that $a_k \in (a - \ve, a + \ve )$.	
	Using \eqref{eq:perturb} and \eqref{eq:contract_star} it is straightforward to 
    verify that
	\begin{align*}
	&\biggl\Vert \sum_{j=0}^{N(\ve_1)-1} \cL_{n - j,n} ( \bar{g}_v ) - \sum_{j=1}^{ N(\ve_1) } \tilde{\cL}^j ( \bar{g}_v )  \biggr\Vert_{L^2(m)} \\
    &\le 
    \sum_{j=0}^{N(\ve_1)-1}
    \sum_{i = 0}^{j}
    \Vert \cL_{ n - j + i + 1, n } ( \cL_{ a_{n - j + i} } - \tilde{\cL} ) \tilde{\cL}^i (\bar{g}_v) \Vert_{L^2(m)}
	\le C N^2(\ve_1)  \ve.
	\end{align*}
    It follows that,
    for all $k \ge N(\ve_1)$,
    \begin{align*}
        &\sup_{\substack{v\in\mathbb R^d \\ | v | = 1 }} |
        \bE[2 \bar{g}_v ( \tilde G_v \circ \tilde{f} - \tilde{G}_v  )] -
        \bE[2 \bar{g}_v (  G_{k+1,v} \circ T_{k+1} - G_{k,v}  ) ]
        | \le C \biggl( \ve_1 + N^2(\ve_1) \ve  \biggr), \\
        &\sup_{\substack{v\in\mathbb R^d \\ | v | = 1 }} 
        | \bE[ \tilde{G}_v^2 - 2 \tilde{G}_v \cdot \tilde{G}_v \circ \tilde{f} + \tilde{G}_v^2 \circ \tilde{f}] - \bE[G_{k,v}^2 - 2  G_{k,v}  \cdot G_{k+1,v} \circ T_{k+1}  + G_{k+1,v}^2 \circ T_{k+1}] | \\
        &\le C \biggl(  \ve_1 + N^2(\ve_1) \ve  \biggr).
    \end{align*}
    Therefore, for all $n \ge 1$, 
	\begin{align*}
		\sup_{\substack{v\in\mathbb R^d \\ | v | = 1 }} |	\bE[  S_{n,v}^2 ] - \bE[ \tilde S_{n,v}^2 ] |  \le C_1 \biggl( C_{\ve_1} + n\ve_1 + n N^2(\ve_1) \ve  \biggr),
	\end{align*}
    where $C_1 > 0$ is a constant independent of $\ve, \ve_1, n$
    and $C_{\ve_1} > 0$ is a constant independent of $\ve, n$.
    First we fix $\ve_1 > 0$ sufficiently small 
	such that
	\[
	C_1 \ve_1 < \kappa / 4.
	\]
	Then we choose $\ve > 0$ sufficiently small such that 
	\[
	C_1 N^2(\ve_1) \ve  < \kappa / 4.
	\]
	It follows that 
	\begin{align*}
		\sup_{\substack{v\in\mathbb R^d \\ | v | = 1 }}
        |	\bE[  S_{n,v}^2 ] - \bE[ \tilde S_{n,v}^2 ] |  \le C_1 C_{\ve_1} + n \kappa / 2,
	\end{align*}
    which, 
    recalling \eqref{eq:variange_mgle}, implies 
	\[
	\min_{\substack{v\in\mathbb R^d \\ | v | = 1 }} \bE[  S_{n,v}^2 ] \ge n \kappa/4,
	\]
	whenever $n$ is sufficiently large. This completes the proof of \eqref{eq:lambda_growth}. \qed

\appendix

\section{Proof of 
Theorem~\ref{thm:decdec}}\label{sec:app_a}

To prove Theorem~\ref{thm:decdec}, we closely follow the coupling approach developed in \cite{KKM19,KL21,KL25}. For completeness, we provide most of the details of the argument.

Define $C_h = 2 e^{ \mathfrak{d} K_2 }$ and
\begin{equation}
	\label{eq:hnk}
	h_n^k(\ell)
	= C_h \bigl( h^k(n + \ell) + h^{k + 1}(n + \ell - 1) + \cdots + h^{k + n}(\ell) \bigr).
\end{equation}
Then 
\begin{align}\label{eq:h_nk}
	h_n^k(\ell) \le  C  \exp( - ( A' / 2 ) \ell^u  ),
\end{align}
where $C$ depends only on $C_h, A, A', u$.

\begin{prop}\label{prop:onedec_prelim} Let $k \in \bZ$. 
Suppose that $\mu$ is a probability measure on $\bar{Y} \subset \bar{X}_k$ with $|\mu|_{\LL} \leq K_2$.
Then, for every $n \geq 0$, the measure $(\bar{T}_{k,k + n - 1})_* \mu$
on $\bar{X}_{k+n}$ is regular with tail bound
$
\frac{1}{2} h^k_n(\ell)
$.
\end{prop}

\begin{proof} Fix $n \geq 0$ and $k \in \bZ$. For each $0 \leq j \le n$, define 
subsets $Y'_j \subset \bar{X}_{k + n - j}$ and $Y_j'' \subset \bar{X}_k$ as follows:
\begin{align*}
	Y'_j
	& = \{ y \in \bar Y \colon \bar{T}_{k + n - j, k + \ell - 1 }(y) \notin \bar Y \text{ for all $n-j <  \ell \le n$}  \}
	,
	\\
	Y_j''
	& = \bar Y \cap   \bar{T}_{k, k + n - j - 1}^{-1} (Y_j')
	\\
	& = \{ y \in \bar Y  \colon \bar{T}_{k, k + n - j - 1}(y) \in \bar Y
	\text{ and $\bar{T}_{k,k + \ell - 1}(y) \notin  \bar Y$  for all $n - j < \ell \le n$} \}.
\end{align*}
Then the sets $Y_j''$ form a partition of $\bar Y$. Moreover, $Y_0' = \bar Y$ and $Y''_0 = \{ y \in \bar Y \colon \bar{T}_{k, k + n - 1 }(y) \in  \bar Y  \}$.
Denoting by $\mu_j$ the restriction of $\mu$ to $Y_j''$, we have 
\begin{align}\label{eq:mu_pw_decomp}
	( \bar{T}_{k,k + n - 1}  )_*\mu
	= \sum_{j=0}^n ( \bar{T}_{k,k + n - 1}  )_*\mu_j,
\end{align}

Next, define
$\nu_j = (( \bar{T}_{k, k + n - j - 1} )_*\mu)|_{  \bar{Y} }$ for $0 \le j \le n$,
which is a measure on $\bar{X}_{k+n-j}$. Note that for all  $B \in \bar{\cB}_{k+n}$,
\begin{align*}
	(  \bar{T}_{k + n - j, k + n - 1} )_* ( \nu_j |_{Y_j'}  ) (B)
	& = \nu_j ( Y_j' \cap \bar{T}_{k + n - j, k + n - 1}^{-1} (B) ) \\
	&= \mu(   \bar{T}_{k, k + n - j - 1}^{-1} (Y_j')   \cap \bar{T}_{k,k + n - 1}^{-1}(B)) \\
	&= (\bar{T}_{k,k + n - 1} )_*\mu_j(B).
\end{align*}
Hence, $(\bar{T}_{k + n - j, k + n - 1} )_* ( \nu_j |_{Y_j'} ) = (\bar{T}_{k,k + n - 1} )_*\mu_j$.
By Proposition~\ref{prop:regular}-(c), $|\nu_j|_{\LL} \leq K_2$. It follows that 
\[
\nu_j(B) \le \nu_j( \bar Y) e^{ \mathfrak{d} | \nu_j |_{\LL } } \bar{m} (B) \le  
e^{ \mathfrak{d} K_2   } \bar{m} (B).
\]
Therefore, the tail of $\nu_j$ is bounded by $e^{  \mathfrak{d}  K_2  } h^{k + n - j}$.

The pushforward measure $(\bar{T}_{k + n - j, k + n - 1} )_* ( \nu_j |_{Y_j'} )$
inherits the tail bound from $\nu_j$ with a time shift, namely it
has tail bound $e^{  \mathfrak{d} K_2  } h^{k + n - j}(\cdot + j)$.
It follows that the measure $(\bar{T}_{k,k + n - 1})_* \mu$ has tail bound
$e^{\mathfrak{d} K_2  } \sum_{j=0}^n h^{k + n - j}(\cdot + j)$, as required. Finally,
$(\bar{T}_{k,k + n - 1})_* \mu$ is a regular measure on $\bar{X}_{k + n}$
by
Proposition~\ref{prop:regular}.
\end{proof}

Next, as in \cite[Proposition~2.1]{KL25}, we use (A5) to derive the following useful property:

\begin{prop}\label{prop:a5}
	There exist $\delta_0 > 0$ and $n_0 \ge 1$ such that 
	\begin{align}\label{eq:a5}
		\inf_{k \in \bZ} \bar m(  \bar{T}_{k, k + n -1 }^{-1}  \bar Y  ) \ge \delta_0
	\end{align}
	whenever $n \ge n_0$. Moreover, $\delta_0$ and $n_0$ depend only on $N$, $n_1,\ldots, n_N$, $\delta_\#$, $\mathfrak{d}$, $K_1, C_h, A, A', u$.
\end{prop}

\begin{proof} As $| \bar{m} |_{\LL} \le K_2$, Proposition~\ref{prop:onedec_prelim} combined with 
	\eqref{eq:h_nk} implies 
	\begin{align*}
		( \bar T_{k, k + j -1})_* \bar m (   \tau_{k + j }  \ge n  )  &\le C \exp( - C' n^u ),
	\end{align*}
	for any $k \in \bZ$ and $j,n \ge 0$, where $C, C' > 0$ depend only on $C_h, A, A', u$.
	In particular, there exists $N_\# \in \bN$ depending only on $C_h, A, A', u$
	such that 
	\begin{align}\label{eq:uniform_tail_bound}
		\sup_{k\in\mathbb Z,j\in\mathbb Z_+} ( \bar T_{k, k + j -1})_* \bar m (   \tau_{k + j }  \ge N_\#  ) \le \frac12. 
	\end{align}
	
	Let 
    \[
    N_1 = \max \{  n_1^2, \ldots, n_N^2 \} \quad \text{and} \quad 
    N_2 = N_1 + N_\#.
    \]
    Since $\text{gcd}(n_1, \ldots, n_N) = 1$, 
    it follows from~\cite{S77} that, for any $n \ge N_1$ 
	there exist integers $a_j \ge 0$ such that
	$
	n = \sum_{j=1}^N a_j n_j
	$.  
    For $n > N_2$, we have
	\[
	(\bar T_{k,  k + n - 1})_* \bar m  = ( \bar T_{ k + n - N_2  ,  k + n - 1 } )_*\nu,
	\]
	where $\nu = ( \bar T_{k,  k +  n - N_2 - 1  })_* \bar m$. By Proposition~\ref{prop:regular}-(c), $\nu$ is a regular measure on $\bar{X}_{k+ n - N_2}$ and 
	$| \nu |_{ \bar{Y} } |_{ \LL} \le K_1$.
	
	Using the regularity of
	$\nu$, for $n > N_2$ we obtain
	\begin{align*}
		(\bar T_{k, k + n - 1  })_* \bar m( \bar Y) 
		&\ge \sum_{ \ell \le N_\#}
		( \bar T_{ k + n - N_2 , k + n - 1  } )_*( \nu  |_{  \{  \tau_{k + n - N_2} = \ell   \}  }  ) ( \bar Y) \\
		&= \sum_{\ell \le N_\#}
		( \bar T_{ k + n - N_2  + \ell   , k + n - 1  } )_*  ( \bar T_{ k + n - N_2 ,   k + n - N_2  +  \ell - 1     } )_* (   \nu  |_{  \{  \tau_{k + n - N_2} = \ell   \}  }   ) ( \bar Y)
		\\
		&\ge  \sum_{\ell \le N_\#}  e^{ -   \mathfrak{d}  K_1} (  \bar T_{ k + n - N_2, k + n - N_2 + \ell - 1  } )_*(  \nu  |_{  \{  \tau_{k + n - N_2} = \ell   \}  }   ) (  \bar Y ) \\
		&\qquad\times (\bar T_{k + n - N_2 + \ell, k + n - 1})_* \bar m( \bar Y).
	\end{align*}
	As in~\cite[Proposition~4.4]{KKM19}, using $\bar m(  \tau_k = n_j ) \ge \delta_\#$
    for $1 \le j \le N$
	and Proposition~\ref{prop:regular}, we find that for any $\ell \le N_\#$,
	\[
	(\bar T_{k + n - N_2 + \ell, k + n - 1})_* \bar m( \bar Y) \ge ( \delta_\# e^{ - (\mathfrak{d} + 1) K_1  } )^{N_2}.
	\]
	Moreover,
	\[
	\sum_{\ell \le N_\#}  (  \bar T_{ k + n - N_2 , k + n - N_2 + \ell - 1  } )_*(  \nu  |_{  \{  \tau_{k + n - N_2} = \ell   \}  }   ) (  \bar Y ) = \nu(\tau_{k + n - N_2 } \le N_\# ).
	\]
	Thus, we arrive at the lower bound 
	\begin{align*}
		(\bar T_{k, k + n - 1  })_* \bar m( \bar Y  ) \ge (\delta_\# e^{- ( \mathfrak{d} + 1 ) K_1  } )^{N_2+1} \nu(\tau_{k + n - N_2 } \le N_\# ).
	\end{align*}
	By~\eqref{eq:uniform_tail_bound}, we have $\nu(\tau_{k + n - N_2 } \le N_\# ) \ge 1/2$. Therefore,
	\[
	( \bar T_{k, k + n - 1  })_* \bar m( \bar Y) \ge   (\delta_\# e^{- ( \mathfrak{d} + 1 ) K_1  } )^{N_2+1} \frac12,
	\]
	whenever $n > N_2$.
\end{proof}

From now on we fix $\delta_0$ and $n_0$ such that \eqref{eq:a5} holds. 

\begin{prop}\label{prop:onedec}
	    Let $k \in \bZ$.
	    Let $\mu$ be a probability measure on $\bar{Y} \subset \bar{X}_k$ with $| \mu |_{\LL} \le K_2$.
		There exists $\theta \in (0,1)$ depending only on $K_1$, $K_2$, $\mathfrak{d}$ and $\delta_0$,
		such that for every $n \geq n_0$,
		\begin{align}\label{eq:onedec}
		( \bar{T}_{k,k + n - 1})_* \mu
		= \theta \bar{m} + (1-\theta) \mu'
		,
		\end{align}
		where $\mu'$ is a regular probability measure
		on $\bar{X}_{k+n}$
		with tail bound $h^k_n$.
\end{prop}

\begin{proof} 
By \cite[Proposition 3.2]{KKM19}, for any $\psi \colon \bar Y \to [0, \infty)$ and $t \in [0, \exp( - |\psi|_{\LL} )]$, we have 
\begin{align}\label{eq:psi_ll}
\biggl| 
\psi - t \int_{ \bar{Y} } \psi \, d \bar m
\biggr|_{\LL} \le \frac{|\psi|_{\LL}}{ 1 - t \exp( |\psi|_{\LL} ) }.
\end{align}
We first record the following consequence of \eqref{eq:psi_ll}.

\medskip
	
	\noindent\textbf{Claim.} 
	There exists $\theta_0 \in (0,1)$ depending only on $K_1, K_2$
	such that for every $\theta' \in [0, \theta_0]$, every measure $\rho$ on $\bar Y$ with $|\rho|_{\LL} \le K_1$ can be written as 
	\[
	\rho = \rho( \bar Y ) \theta' \bar m + \rho',
	\]
	where $\rho'$ is a measure on $\bar Y$ with $|\rho'|_{\LL} \le K_2$.
	
	\medskip
	
	\noindent\textit{Proof of the claim.} Let
	\[
	\psi' = \psi - \theta' \rho( \bar Y),
	\]
	where $\psi = d \rho / d \bar{m}$.	Then \eqref{eq:psi_ll} implies 
	\[
	|\psi'|_{\LL} \le \frac{K_1}{1 - \theta' \exp(K_1)} \quad \text{if $\theta' \le \exp(-K_1)$.}
	\]
	Therefore,
	\[
	|\psi'|_{\LL} \le K_2 \quad \text{if $\theta' \le \exp(-K_1)\biggl( 1 - \frac{K_1}{K_2} \biggr)$.}
	\]
	So, we may choose $\theta_0 = \exp(-K_1) ( 1 - K_1 / K_2 )$.
    \hfill
	$\qed_{\text{Claim}}$
	
	\medskip

    Next, we prove the decomposition \eqref{eq:onedec}. To this end, let 
	\[
	\rho_{k,n} = [  ( \bar{T}_{k,k + n - 1})_* \mu ] |_{  \bar Y },
	\]
	where $n \ge n_0$. By Proposition~\ref{prop:regular}, we have $|\rho_{k,n}|_{\LL} \le K_1$.
	Moreover, using \eqref{eq:a5} we obtain 
	\[
	\rho_{k,n}( \bar{Y} ) \ge \delta_0 e^{  -  \mathfrak{d} K_2 }.
	\]
	Set $\theta = \min \{ \theta_0  \delta_0 e^{  -   \mathfrak{d}  K_2  }, 1/2 \}$. Then 
	\[
	\rho_{k,n} =  \theta \bar m + \rho'_{k,n},
	\]
	where $\rho'_{k,n}$ is a measure on $\bar Y$ with $|\rho'_{k,n}|_{\LL} \le K_2$.
	
	Further, define
	\[
	\mu'
	= (1-\theta)^{-1} \bigl( (\bar{T}_{k,k + n - 1})_* \mu - \theta \bar m \bigr)
	= (1-\theta)^{-1} \Bigl( \rho'_{k,n} + \bigl((\bar{T}_{k,k+n-1})_*
	\mu \bigr)\big|_{ \bar{X}_{k+n} \setminus \bar Y} \Bigr)
	.
	\]
	Then $\mu'$ is a probability measure 
	on $\bar{X}_{k + n}$
	and
	$(\bar{T}_{k,k+n-1})_* \mu = \theta \bar m + (1-\theta) \mu'$.
	By Proposition~\ref{prop:regular},
	$\rho'_{k, n}$ and $\bigl((\bar{T}_{k,k+n-1})_* \mu \bigr)\big|_{\bar{X}_{k+n} \setminus
    \bar{Y} }$
	are regular measures, and hence so is $\mu'$. Finally,
	\[
	\mu' \leq (1-\theta)^{-1} (\bar{T}_{k,k+n-1})_* \mu \le 2 (\bar{T}_{k,k+n-1})_* \mu.
	\]
	Therefore, $\mu'$ has tail bound $h_n^k(\ell)$ by 
    Proposition~\ref{prop:onedec_prelim}.
\end{proof}

Let $\hr (0) = 1$ and 
\[
\hr (n) = \min \{ 1, r(1), \ldots, r(n)\}, \quad n \ge 1.
\]
Then, $\hr$ is nonincreasing and $\hr(1) = 1$. Note that
for a probability measure, any tail bound $r$ can be replaced by $\hat r$.
Define $\hh^k_n$ in an analogous way.

Let $\xi_1, \xi_2, \ldots$ be random variables on a probability space $(\Omega, \cF, \bP)$
with values in $\{n_0, n_0+1,\ldots\}$, such that
for all $\ell \geq n_0$,
\begin{equation}
	\label{eq:PXj}
	\begin{aligned}
		\bP(\xi_1 \geq \ell)
		&= \hr(\ell - n_0),
		\\
		\bP(\xi_{j+1} \geq \ell \mid \xi_1, \ldots, \xi_j)
		& = \hat{h}^{k+ \xi_1 + \cdots + \xi_{j-1}}_{\xi_j}(\ell - n_0)
		\quad \text{ for } j \geq 1
		.
	\end{aligned}
\end{equation}
Let $\tau$ be a geometrically distributed random variable with parameter $\theta$,
independent of $\{\xi_j\}$. That is,
\[
\bP(\tau = \ell) = (1-\theta)^{\ell-1} \theta \quad \forall \ell \in \{1,2,\ldots\}.
\]
Define
\[
S
= \xi_1 + \cdots + \xi_\tau
.
\]


\begin{lemma}
	\label{lem:proproprobab}
    Let $k \in \bZ$.
	Suppose that $\mu$ is a regular probability measure on $\bar{X}_k$
	with tail bound $r$.
	Then there exists a decomposition
	\[
	\mu = \sum_{n=1}^\infty \bP(S = n) \mu_n,
	\]
	where $\mu_n$ are probability measures on $\bar X_k$ such that $(\bar T_{k,k+n-1})_* \mu_n = \bar m$.
\end{lemma}

\begin{proof}
Having established Proposition \ref{prop:onedec}, the result follows by repeating the arguments in
\cite[Lemmas~4.3 and 4.5]{KL25}. We omit the details.
\end{proof}

To complete the proof of Theorem~\ref{thm:decdec}, it remains to estimate the tail
probabilities 
$\bP( S \ge n )$ in Lemma~\ref{lem:proproprobab}.

\begin{prop}
	For all $n \ge 1$,
	\[
	\bP( S \ge n ) \le C \exp(  - C' n^u ),
	\]
	where $C, C' > 0$ are constants depending only on $C_r, C_r', A, A', u, K_1, K_2$ and 
	$\mathfrak{d}$, $K$, $\lambda$, $\delta_\#$, $N$, $n_1,\ldots, n_N$.
\end{prop}

\begin{proof} 
	Denote by $C,C_1$ and $C',C_1'$  various positive constants
	depending only on 
	$C_r, C_r', A, A', u, K_1, K_2$ and the system 
	constants $\mathfrak{d}$, $K$, $\lambda$, $\delta_\#$, $N$, $n_1,\ldots, n_N$.

    By \eqref{eq:hnk} and  
    \eqref{eq:rn}, for $\ell \ge 0$
	\begin{align*}
		&\bP(\xi_1 \ge \ell) \le C_r  C \exp( - C_r' \ell^u ), \\
		&\bP(\xi_{j+1} \geq \ell \mid \xi_1, \ldots, \xi_j) \le C \exp( - C' \ell^u ).
	\end{align*}
	Consequently, there exist $C_1, C_1' > 0$ such that for every $t \in [0, \infty)$ and every $j \ge 1$,
	\begin{align*}
		\bP(  \xi_{j} \ge t \mid \xi_1, \ldots, \xi_{j-1}  ) \le C_1 \exp( - C_1' t^u ).
	\end{align*}
	Then, arguing as in \cite[Proposition~4.11]{KKM19}, it follows that $S_k = \sum_{j=1}^k \xi_j$ 
	satisfies
	\[
	\bP( S_k \ge t )  \le (1 + \beta D )^{k} e^{ - \beta t^u },
	\]
	whenever $t \in [0, \infty)$, $k \ge 1$ and $\beta \le C_1' / 2$. Here 
	\[
	D = C_1 u \int_0^\infty s^{ u - 1 } e^{ - \frac12 C_1' s^u } \, ds.
	\]
	Set 
	\[
	\beta = 
    \min \biggl\{ 
    \frac12 \frac{  (1- \theta)^{ - \frac12 } - 1 }{D},
    \frac{C_1'}{2} \biggl\}, 
	\]
	so that
	\[
	1 + \beta D \le (1- \theta)^{ - \frac12 }.
	\]
	Then,
	\begin{align*}
		\bP(S \ge n) = \sum_{k \ge 1} \bP( S_k \ge n ) \bP(\tau = k) 
		\le  e^{ - \beta n^u }
		\theta 
		(1 - \theta)^{-1} \sum_{k=1}^\infty (1- \theta)^{ k / 2 } \le Ce^{ - \beta n^u },
	\end{align*}
    as required.
\end{proof}

\subsection*{Acknowledgments} 
JL was supported by the JSPS via the project
LEADER. HT was supported by the JSPS KAKENHI 25K21999 and 26H02003. \\

\noindent{\bf Data Availability} This article has no associated data and material.\\

\noindent{\bf Declarations}\medskip

\noindent{\bf Conflict of interest} We have no Conflict of interest.


\begin{thebibliography}{99}

\bibitem{AHNT15}
Aimino, R., Hu, H., Nicol, M.,  T\"or\"ok, A., Vaienti, S.: 
Polynomial loss of memory for maps of the interval with a neutral fixed point. 
Discrete Contin. Dyn. Syst. \textbf{35}, 793--806 (2015)

\bibitem{AimRou} Aimino, R., 
Rousseau, J.: Concentration inequalities for sequential dynamical systems of the unit interval. 
Ergod. Theory Dyn. Syst.  {\bf 36}, no. 8, 2384--2407 (2016)

\bibitem{AB26}
 Alves, J.F., Bahsoun, W.: Decay of correlations for partially hyperbolic skew-products.
arXiv preprint arXiv:2607.21516 (2026)

\bibitem{CNW23}
 Chung, Y.M., Nakano, Y., Wittsten, J.: Quenched limit theorems for random $U(1)$ extensions of expanding maps.
Discrete Contin. Dyn. Syst. \textbf{43}, 338--377 (2023)

\bibitem{CR07}
Conze, J.-P., Raugi, A.: Limit theorems for sequential expanding dynamical systems on $[0, 1]$. Contemp. Math. {\bf 430}, 89--121 (2007)

\bibitem{CN24}
Crimmins, H., Nakano, Y.: A spectral approach to quenched linear and higher-order response for partially hyperbolic dynamics. Ergod. Theory Dyn. Syst. \textbf{44}, 1026--1057 (2024)


\bibitem{CC25}
Cui, H., Wang, C.: Maximal large deviations for sequential dynamical systems. J. Stat. Phys. \textbf{192}, (2025)

\bibitem{DM22}
Dedecker, J., Merlev\'{e}de, F., Rio, E.: Rates of convergence in the central limit theorem for martingales in the non stationary setting.
Ann. Inst. H. Poincar\'{e} Probab. Statist., \textbf{58} (2022)

\bibitem{DL25}
Demers, M., Liverani, C.: Central limit theorem for sequential dynamical systems. arXiv preprint arXiv:2502.07765 (2025)

\bibitem{DH25}
Dolgopyat, D., Hafouta, Y.: Berry Esseen theorems for sequences of expanding maps. Probab. Theory Rel. Fields \textbf{193}, 1075--1119 (2025)



\bibitem{DGS23}
Dragi\v cevi\'c, D., Giulietti, P., Sedro, J.:
Quenched linear response for smooth expanding on average cocycles.
Commun. Math. Phys. \textbf{399}, 423--452 (2023)

\bibitem{DGTS25} Dragi\v cevi\' c, D.,  Gonz\'{a}lez-Tokman, C., Sedro, J.: Linear response for random and sequential intermittent maps. J. Lond. Math. Soc. \textbf{111},  e70150, 39pp (2025)

\bibitem{DH26}
Dragi\v cevi\'c, D., Hafouta, Y.: Quenched and annealed linear response for some partially hyperbolic skew products. arXiv preprint arXiv:2604.20402 (2026)

\bibitem{DL26}
Dragi\v cevi\'c, D., Lepp\"anen, J.:
Regularity of the variance in quenched CLT for random intermittent dynamical systems.
Ergod. Theory Dyn. Syst. \textbf{46}, 1127--1153 (2026)


\bibitem{FV22}
Fleming--V\'{a}zquez, N.: 
Functional correlation bounds and optimal iterated moment bounds for slowly-mixing nonuniformly hyperbolic maps.
Commun. Math. Phys. \textbf{391}, 173--198 (2022)

\bibitem{FFV17}
Freitas, A.C., Freitas, J., Vaienti, S.: 
Extreme value laws for non stationary processes generated by sequential and random dynamical systems. 
Ann. Inst. H. Poincar\'{e} Probab. Statist. \textbf{53}, 1341--1370 (2017)

\bibitem{FFV18}
Freitas, A.C., Freitas, J., Vaienti, S.: 
Extreme Value Laws for sequences of intermittent maps. 
Proc. Am. Math. Soc. \textbf{146}, 2103--2116 (2018)

\bibitem{GL26}
Galatolo, S., Lucarini, V.: A mathematical framework for linear response theory for nonautonomous systems. arXiv preprint arXiv:2603.19509 (2026)

\bibitem{GMS18}
Gallou\"et, T., Mijoule, G., Swan, Y.:
Regularity of solutions of the Stein equation and rates in the multivariate central limit theorem. arXiv preprint
arXiv:1805.01720 (2018)

\bibitem{GO13}
Gupta, C., Ott, W., T\"or\"ok, A.: Memory loss for time-dependent piecewise expanding systems in higher dimension. Math. Res. Lett. {\bf 20}, 
141--161 (2013)

\bibitem{H22}
Hafouta, Y.: Limit theorems for random non-uniformly expanding or hyperbolic maps with exponential tails. 
Ann. H. Poincar\'{e} \textbf{23}, 293--332 (2022)

 \bibitem{HI05} Hamachi, T., Inoue, K.: Embedding of shifts of finite type into the Dyck shift. Monatsh. Math. {\bf 145}, 107--129 (2005)

\bibitem{HNTV17}
Haydn, N., Nicol, M., T\"or\"ok, A.,  Vaienti, S.: Almost sure invariance principle for sequential and non-stationary dynamical systems. Trans. Am. Math. Soc.  \textbf{369}, 5293--5316 (2017)

\bibitem{Hof79} Hofbauer, F.: On intrinsic ergodicity of
    piecewise monotonic transformations with positive entropy I, II.
    Israel J. Math. {\bf 34}, 
    213--237 (1979), {\bf 38}, 
    107--115 (1981)


\bibitem{KKM19} Korepanov, A., Kosloff, Z., Melbourne, I.:
Explicit coupling argument for non-uniformly hyperbolic transformations.
Proc. Roy. Soc. Edinburgh Sect. A {\bf 149}, no. 1, 101--130  (2019)

\bibitem{KL21} Korepanov, A.,  Lepp\"anen, J.: Loss of memory and moment bounds for nonstationary intermittent dynamical systems. Commun. Math. Phys. {\bf 385}, 905--935 (2021)

\bibitem{KL25} Korepanov, A.,  Lepp\"anen, J.: Improved polynomial rates of memory loss for nonstationary intermittent dynamical systems. 
 Phys. D {\bf 483}, Paper No. 134939. 37 (2025)

\bibitem{Kri74} 
 Krieger, W.: On the uniqueness of the equilibrium state. Math. Syst. Theory {\bf 8}, 97--104 (1974/75)

\bibitem{LNN25} Lepp\"anen, J., Nakajima, Y., Nakano, Y.: 
Error bounds in a smooth metric for Brownian approximation of dynamical systems via Stein's method.
J. Stat. Phys. \textbf{192}, no. 12, Paper No. 168, 45 pp (2025)

\bibitem{LNN26} Lepp\"anen, J., Nakajima, Y., Nakano, Y.: 
Functional correlation bound for random Lasota--Yorke maps with holes and its applications to conditional normal approximations.
arXiv preprint arXiv:2604.17203 (2026)

\bibitem{LS20} Lepp\"anen, J., Stenlund, M.: Sunklodas' approach to normal approximation for time-dependent dynamical systems. 
J. Stat. Phys. {\bf 181}, no. 5, 1523--564 (2020)

\bibitem{LW24}
Liu, Z., Wang, Z.: Wasserstein convergence rates in the invariance principle for sequential dynamical systems. Nonlinearity \textbf{37}, 
1172--1191 (2024) 

\bibitem{MO14} Mohapatra, A., Ott, W.: Memory loss for nonequilibrium open dynamical systems.  Discrete Contin. Dyn. Syst. \textbf{34}, 
3747--3759 (2014)

\bibitem{NW15}
Nakano, Y., Wittsten, J.: On the spectra of quenched random perturbations of partially expanding maps on the torus. Nonlinearity \textbf{28}, 951--1002 (2015)

\bibitem{NTV18} Nicol, M., T\"or\"ok, A., Vaienti, S.:
Central limit theorems for sequential and random intermittent dynamical systems.
Ergod. Theory Dyn. Syst.  \textbf{38},  1127--1153
 (2018)

\bibitem{STYY} Saiki, Y., Takahasi, H., Yamamoto, K., Yorke, J.A.: The dynamics of the heterochaos baker maps. Chinese Annals of Mathematics, Proceedings of Shishikura 60, to appear 

\bibitem{STY21}
Saiki, Y., Takahasi, H., Yorke, J.A.: Piecewise linear maps with heterogeneous chaos. Nonlinearity {\bf 34}, 5744--5761 (2021)


\bibitem{S77} Selmer, E.S.: On the linear Diophantine problem of Frobenius.
J. Reine Angew. Math. {\bf 293/294}, 1--17 (1977)

\bibitem{SYZ13}
Stenlund, M., Young, L.-S., Zhang, H.: Dispersing billiards with moving scatterers. Commun. Math. Phys. \textbf{322}, 909--955 (2013) 

\bibitem{S22}
Su, Y.: Vector-valued almost sure invariance principles for (non)stationary and random dynamical systems.
Trans. Am. Math. Soc. \textbf{375},   4809--4848 (2022)

\bibitem{T25}
Takahasi, H.: Exponential mixing for heterochaos baker maps and the Dyck system. J. Dyn. Differ. Equ. {\bf 37}, no. 3, 2409--2436 (2025)

\bibitem{TT25}
Takahasi, H., Tsujii, M.: Polynomial rate of mixing for the heterochaos baker maps with mostly neutral center. J. Dyn. Differ. Equ. {\bf 38},  no. 2, 
1099--1124 (2026)

\bibitem{TY23} Takahasi, H., Yamamoto, K.: Heterochaos baker maps and the Dyck system: maximal entropy measures and a mechanism for the breakdown of entropy approachability. Proc. Am. Math. Soc. {\bf 153}, no. 8, 3335--3350 (2025)

\bibitem{You98} 
Young, L.-S.:
Statistical properties of dynamical systems with some hyperbolicity.
Ann. Math. {\bf 147}, 585--650 (1998)


\end{thebibliography}
\end{document}